\documentclass[12pt,reqno]{amsart}

\usepackage{amssymb,amsmath,graphicx,
	amsfonts,euscript}
\usepackage{color}
\usepackage{mathrsfs}
\allowdisplaybreaks

\numberwithin{equation}{section}

\newcommand{\Z}{{\mathbb Z}}
\newcommand{\R}{{\mathbb R}}

\def\p{\partial}

\def\om{\omega}

\newtheorem{Thm}{Theorem}[section]
\newtheorem{thm}[Thm]{Theorem}
\newtheorem{coro}[Thm]{Corollary}
\newtheorem{rem}[Thm]{Remark}
\newtheorem{lem}[Thm]{Lemma}
\newtheorem{prop}[Thm]{Proposition}

\newcommand{\beq}{\begin{equation}}
\newcommand{\eeq}{\end{equation}}
\newcommand{\ben}{\begin{eqnarray}}
\newcommand{\een}{\end{eqnarray}}
\newcommand{\beno}{\begin{eqnarray*}}
\newcommand{\eeno}{\end{eqnarray*}}

\begin{document}

\title[Stability of Couette flow]{Stability of Couette flow for 2D Boussinesq system with vertical dissipation}

\author[Wen Deng, Jiahong Wu and Ping Zhang]{Wen Deng$^{1}$, Jiahong Wu$^{2}$ and Ping Zhang$^{3}$}

\address{$^{1}$ Academy of Mathematics and Systems Science and Hua Loo-Keng Key Laboratory of Mathematics, Chinese Academy of Sciences,  Beijing 100049, China}

\email{dengwen@amss.ac.cn}

\address{$^2$ Department of Mathematics, Oklahoma State University, Stillwater, OK 74078, United States}

\email{jiahong.wu@okstate.edu}

\address{$^3$ Academy of
Mathematics $\&$ Systems Science and  Hua Loo-Keng Key Laboratory of
Mathematics, The Chinese Academy of Sciences, China, and School of Mathematical Sciences, University of Chinese Academy of Sciences, Beijing 100049, China }

\email{zp@amss.ac.cn}

\vskip .2in
\begin{abstract}
	This paper establishes the nonlinear stability of the Couette flow for the 2D Boussinesq
	equations with only vertical dissipation. The Boussinesq equations concerned here model
	buoyancy-driven fluids such as atmospheric and oceanographic flows. Due to the presence of
	the buoyancy forcing, the energy of the standard Boussinesq equations could grow in time.
	It is the enhanced dissipation created by the linear non-self-adjoint operator $y\partial_x -\nu\partial_{yy}$
	in the perturbation equation that makes the nonlinear stability possible. When the initial perturbation from the Couette flow $(y, 0)$ is no more than the viscosity to a suitable power (in the Sobolev space $H^b$ with $b>\frac43$), we prove that the solution of the
	2D Boussnesq system with only vertical dissipation on $\mathbb T\times \mathbb R$ remains close to the Couette at the same order. A special consequence of this result is the stability
	of the Couette for the 2D Navier-Stokes equations with only vertical dissipation.
	\end{abstract}
	
\maketitle

\section{introduction}

\vskip .1in
The Boussinesq system reflects the basic physics laws obeyed by buoyancy-driven fluids. It is one of the most frequently used models for atmospheric and oceanographic flows and serves as the centerpiece in the study of the Rayleigh–Bénard convection (see, e.g., \cite{ConD, DoeringG,Maj,Pe}). The Boussinesq equations are mathematically significant.
The 2D Boussinesq equations serve as a lower dimensional model
of the 3D hydrodynamics equations. In fact, the 2D Boussinesq equations retain some key features of the 3D Euler and Navier-Stokes equations such as the vortex stretching mechanism.  The inviscid 2D Boussinesq equations can be identified as the Euler equations for the 3D axisymmetric swirling flows \cite{MaBe}. Furthermore, the Boussinesq equations have some special characteristics of their own and offer many opportunities for new discoveries.

\vskip .1in
Due to their broad physical applications and mathematical significance,
the Boussinesq equations have recently attracted
considerable interests. Two fundamental problems, the global regularity problem and the stability problem, have been among the main driving
forces in advancing the mathematical theory on the Boussinesq
equations. Significant progress has been made on the global regularity of the 2D Boussinesq equations, especially those with only partial or fractional dissipation or no dissipation
at all.
Our attention here will be focused on the stability problem. The study of the stability problem  on two physically important steady states
has gained strong momentum. The first steady state is the hydrostatic equilibrium, which  is a prominent topic in fluid dynamics and astrophysics. Understanding this stability problem may help gain insight into some weather phenomena. Important progress has been made on the stability and
large-time behavior (\cite{CCL, DWZZ, TWZZ, Wan}). The second steady state is the shear flow, which is the focus of this paper.
The aim here is to fully understand the stability of perturbations near the Couette
flow and their large-time behavior. Our consideration will cover both the Boussinesq equations
with full dissipation and the Boussinesq equations with only vertical dissipation. Our emphasis
is on the case when the dissipation is degenerate and only in the vertical direction.

\vskip .1in
The 2D Boussinesq system with full dissipation is given by
\begin{equation}\label{eq.B}
\begin{cases}
\p_tu+(u\p_x +v\p_y) u=-\p_x p+\nu\Delta u, \\
\p_tv+(u\p_x+v\p_y) v=-\p_y p+\nu\Delta v+\theta,\\
\p_xu+\p_yv=0,\\
\p_t\theta+(u\p_x +v\p_y)\theta=\mu\Delta \theta,
\end{cases}
\end{equation}
where ${\bf u}=(u,v)$ denotes the 2D velocity field, $p$ the pressure, $\theta$ the temperature,
$\nu$ the viscosity and $\mu$ the thermal diffusivity. The first three equations in (\ref{eq.B}) are the incompressible Navier-Stokes equation
with buoyancy forcing in the vertical direction. The last equation is a balance of the temperature
convection and diffusion. The spatial domain $\Omega$ here is taken to  be
$$
\Omega={\mathbb T}\times\R
$$
with $\mathbb T=[0, 2\pi]$ being the periodic box  and $\mathbb R$ being the whole line.  In suitable physical regimes or under suitable scaling, the Boussinesq
equations may involve only vertical dissipation (\cite{MajG}),  namely
\begin{equation}\label{eq.BV}
\begin{cases}
\p_tu+(u\p_x +v\p_y) u=-\p_x p+ \nu\p_{yy} u,\\
\p_tv+(u\p_x+v\p_y) v=-\p_y p+ \nu \p_{yy} v + \theta,\\
\p_xu+\p_yv=0,\\
\p_t\theta+(u\p_x +v\p_y)\theta=\mu\p_{yy}\theta.
\end{cases}
\end{equation}
Cao and Wu previously examined the 2D Boussinesq system with only vertical dissipation
and established its global regularity \cite{CaoWu1}.
The Couette flow,
$$
{\bf u}_{sh}=(y,0), \quad p_{sh}=0, \quad \theta_{sh}=0,
$$
is clearly a stationary solution of \eqref{eq.B} and also of (\ref{eq.BV}). Our goal is to understand the stability and large-time behavior of perturbations near the Couette flow.
The perturbations
\begin{equation*}\label{}
\widetilde u=u-y,\qquad \widetilde v=v,\qquad \widetilde p=p,\qquad \widetilde \theta=\theta,
\end{equation*}
satisfy, in the case of full dissipation,
\begin{equation*}\label{}
\begin{cases}
\p_t\widetilde u+y\p_x\widetilde u+\widetilde v+(\widetilde {\bf u}\cdot\nabla) \widetilde u =-\p_x\widetilde p+\nu \Delta\widetilde u, \\
\p_t\widetilde v+y\p_x\widetilde v+(\widetilde {\bf u}\cdot\nabla )\widetilde v=-\p_y\widetilde p+\nu\Delta \widetilde v+\widetilde \theta ,\\
\p_x \widetilde u+\p_y\widetilde v=0,\\
\p_t\widetilde\theta+y\p_x\widetilde \theta+(\widetilde{\bf u}\cdot\nabla) \widetilde \theta=\mu\Delta\widetilde\theta.
\end{cases}
\end{equation*}
The corresponding perturbed vorticity near the steady vorticity $\om_{sh} =-1$
$$\widetilde \omega=\p_x\widetilde v-\p_y\widetilde u$$
verifies, together with $\widetilde \theta$, the following system
\begin{equation}\label{S1eq.p}
\begin{cases}
\p_t\widetilde \omega+y\p_x\widetilde \omega+(\widetilde{\bf u}\cdot\nabla) \widetilde \omega =\nu \Delta\widetilde \omega+\p_x\widetilde \theta ,\\
\p_t\widetilde\theta+y\p_x\widetilde \theta+(\widetilde{\bf u}\cdot\nabla) \widetilde \theta=\mu\Delta\widetilde\theta,\\
\widetilde{\bf u}=-\nabla^\perp(-\Delta )^{-1}\widetilde\omega.
\end{cases}
\end{equation}
In the case when there is only vertical dissipation,
 the vorticity perturbation $\widetilde \om$ and
the temperature perturbation $\widetilde \theta$ satisfy
\begin{equation}\label{vt}
\begin{cases}
\p_t\widetilde \omega+y\p_x\widetilde \omega+(\widetilde{\bf u}\cdot\nabla) \widetilde \omega =\nu \p_{yy}\widetilde \omega+\p_x\widetilde \theta ,\\
\p_t\widetilde\theta+y\p_x\widetilde \theta+(\widetilde{\bf u}\cdot\nabla) \widetilde \theta=\mu\p_{yy}\widetilde\theta,\\
\widetilde{\bf u}=-\nabla^\perp(-\Delta )^{-1}\widetilde\omega.
\end{cases}
\end{equation}

\vskip .1in
The stability problem proposed for study here on (\ref{S1eq.p}) or (\ref{vt}) is not trivial.
Due to the presence of the buoyancy forcing term, the Sobolev norms or even the $L^2$-norm of the
velocity field  could grow in
time if the two linear terms $y \p_x \widetilde \om$ and $y \p_x \widetilde \theta$ were not included in (\ref{S1eq.p}) or (\ref{vt}).
In fact, Brandolese and Schonbek have shown in \cite{BrS} that the $L^2$-norm of the
velocity to the Boussinesq system with full viscous dissipation and thermal diffusion can grow
in time even for very nice initial data (say, data that are smooth, fast  spatial
decaying and small in some strong norm). The stability of the Couette flow on (\ref{S1eq.p})
and (\ref{vt}) is only possible because of the enhanced dissipation generated by
the non-self-adjoint operator $y\p_x-\nu\p_{yy}$, which is the linear part of the
system (\ref{vt}). Even though the linear operator $y \p_x - \nu \p_{yy}$ involves
only vertical dissipation,
the non-commutativity between its real part and imaginary part actually creates smoothing effect in the horizontal direction, a phenomenon that is called the hypoellipticity.
Operators of this type are investigated by H\"ormander \cite{Ho}.
For the standard heat equation $\p_t f =\nu \Delta f$, the dissipation time scale is $O(\nu^{-1})$ while, for the drift diffusion equations
$$
\p_t f + y\p_x f =\nu \Delta f \quad \mbox{and}\quad
\p_t f + y\p_x f =\nu \p_{yy} f,
$$
the dissipation time scale is $O(\nu^{-\frac13})$, which is much faster than $O(\nu^{-1})$ for small $\nu$. A more detailed explanation will be provided later.  This enhanced dissipation
effect plays an extremely important role in the stability problem studied here.

\vskip .1in
The phenomenon of enhanced dissipation has been widely observed and studied
in physics literature (see, e.g., \cite{Bern, Lat, Kevin, Rhi}). It has recently
attracted enormous attention from the mathematics community and significant progress has been
made. One of the earliest rigorous results on the enhanced dissipation is obtained by
Constantin, Kiselev, Ryzhik and Zlatos on the enhancement of diffusive mixing \cite{Con}. Many remarkable results have since been established. In particular, the stability of the
shear flows to passive scale equations and to the Navier-Stokes equations has been
intensively investigated in a sequence of outstanding papers (see, e.g., \cite{BedCoti0, BedCoti1,
Bed1, Bed2, Bed3, Bed4, Mas0, Mas1, Wei0, Wei1}).

\vskip .1in
The study of the stability problem on the Boussinesq system near the shear flow is very recent.
The work of Tao and Wu \cite{Tao} was able to establish the stability and the enhanced dissipation
phenomenon for the linearized 2D Boussinesq equations with only vertical dissipation, using the method of hypocoercivity introduced by C. Villani \cite{Vi}.
The Boussinesq system is different from the Navier-Stokes equations. The buoyancy
force in the velocity equation could drive the growth of the energy and more generally the
growth of the Sobolev norms. In addition, when there is only vertical dissipation, the
control of the nonlinear terms becomes much more difficult. New techniques and estimates have
to be created in order to handle the degenerate dissipation. It also appears that no previous work
has handled the degenerate case.  Since the Boussinesq system reduces to the Navier-Stokes
equation when $\theta$ is identically zero, the stability results presented in this paper fill the gap on the Navier-Stokes equations with only vertical dissipation.

\subsection{Results}

\vskip .1in
We present three main results. The first result is on the linearized Boussinesq equations
with either full dissipation or with only vertical dissipation. The upper bounds
are explicit and sharp.
The second result assesses the nonlinear stability and
large-time behavior of the Boussinesq system with full dissipation. The third stability result is for the case with only vertical dissipation. Both nonlinear stability results are presented
in order to make a direct comparison between the full dissipation and the degenerate dissipation cases.

\vskip .1in
For notational convenience, we shall write $\omega$ for $\widetilde \omega$ and $\theta$ for $\widetilde \theta$ from now on.
To explain
the linear stability result, we rewrite the equation for both the full dissipation case and
the vertical dissipation case as
\begin{equation}\label{both}
\begin{cases}
\p_t\omega+y\p_x \omega=\nu (\sigma \p_{xx} + \p_{yy}) \omega+\p_x \theta ,\\
\p_t\theta+y\p_x \theta=\mu (\sigma \p_{xx} + \p_{yy}) \theta,\\
\omega|_{t=0}=\omega^{(0)},\quad \theta|_{t=0}=\theta^{(0)}.
\end{cases}
\end{equation}
$\sigma=1$ corresponds to the full dissipation case while $\sigma=0$ to the vertical
dissipation case. To help understand the stability results presented below, we
explicitly solve the linear equation
\beq \label{xxx}
\p_t F +y\p_x F = \nu (\sigma \p_{xx} + \p_{yy}) F, \qquad F(x,y, 0) =F_0(x,y).
\eeq
Taking the Fourier transform yields
$$
\p_t \widehat F - k \p_\xi \widehat F = -\nu (\sigma k^2 + \xi^2) \widehat F, \qquad \widehat F(k,\xi, 0)= \widehat F_0(k,\xi),
$$
where the Fourier transform is given by
$$
\widehat{F}(k, \xi) =\mathcal F F= \int_{y\in \mathbb R} \int_{x\in \mathbb T} F(x,y) e^{-i (kx + \xi y)} \,dxdy.
$$
Making the natural change of variables
$$
\eta := \xi + k t, \qquad H(k, \eta, t) := \widehat F(k, \xi, t) ,
$$
we find that
$$
\p_t H(k, \eta, t) = - \nu( \sigma k^2 + (\eta -k t)^2)\, H(k, \eta, t), \qquad H(k, \eta, 0) = \widehat F_0(k, \eta).
$$
Integrating in time yields
$$
 H(k, \eta, t) = \widehat F_0(k, \eta) e^{- \nu \int_0^t( \sigma k^2 + (\eta -k \tau)^2)\,d\tau}.
$$
Therefore,
\ben
\widehat F(k, \xi,t) &=& H(k,\eta,t) = \widehat F_0(k, \xi + k t)\, e^{-\nu\int_0^t \sigma k^2 + (\xi + k(t-\tau))^2\,d\tau} \notag\\
&=& \widehat F_0(k, \xi + k t)\, e^{-\nu( \sigma k^2 + \xi^2) t}\, e^{-\frac13 \nu k^2 t^3 - \nu k\xi t^2}. \label{rep}
\een
This explicit representation reflects the enhanced dissipation. Even when there is only vertical dissipation, namely $\sigma =0$, the solution is dissipated and regularized in both directions. The dissipation time scale is $O(\nu^{-\frac13})$, which is much faster than the standard
dissipation time scale $O(\nu^{-1})$.  Clearly the dissipation rate is inhomogeneous and depends on the frequencies $k$.

\vskip .1in
Solutions of (\ref{both}) share the same properties as that of (\ref{xxx}). The linear stability results on (\ref{both}) are stated in Proposition \ref{linst} and Proposition \ref{linst2}. To make the statement precise, we define, for $f=f(x,y)$ with $(x,y)
\in \mathbb T\times \mathbb R$ and $k\in \mathbb Z$,
$$
f_k(y):=\frac{1}{2\pi}\int_{\mathbb T}f(x,y)e^{-ikx}dx.
$$
In addition, we write $D=\frac1i\p$. The linear stability result for (\ref{both}) can then be stated as follows.

\begin{prop} \label{linst}
	Let $(\omega,\theta)$ be the solution to \eqref{both} with initial data $(\omega^{(0)},\theta^{(0)})$. There exist constants $c>0$, $C>0$ such that for any $k\in\Z$, $t>0$,
	\begin{equation}\label{S1.1prop1eq1}
	\begin{split}
	&\|\theta_k(t)\|_{L_y^2}\leq C\|\theta_k^{(0)}\|_{L_y^2}e^{-c\mu^\frac13|k|^\frac23t},\\
	&\|\omega_k(t)\|_{L_y^2}\leq C\big(\|\omega_k^{(0)}\|_{L_y^2}+(\nu\mu)^{-\frac16}|k|^\frac13\|\theta^{(0)}_k\|_{L_y^2}\big)e^{-c\nu^\frac13|k|^\frac23t}.
	\end{split}
	\end{equation}
	More generally, assuming that $\nu\lesssim\mu$, for $N\geq0$, there exist $c_N>0$ and  $C_N>0$ such that for any $k\in\Z$, $t>0$,
	\begin{equation}\label{S1.1prop1eq2}
	\begin{split}
	\|D_y^N\theta_k(t)\|_{L_y^2}&\leq C_Ne^{-c_N\mu^\frac13|k|^\frac23 t}\Big(\|D_y^N\theta_k^{(0)}\|_{L_y^2}+(\mu^{-1}|k|)^{\frac{N}3}\|\theta_k^{(0)}\|_{L_y^2}\Big), \\
	\|D_y^N\omega_k(t)\|_{L_y^2}&\leq C_Ne^{-c_N\nu^\frac13|k|^\frac23t}\Big(\|D_y^N\omega_k^{(0)}\|_{L_y^2}+(\nu\mu)^{-\frac16}|k|^\frac13\|D_y^N\theta_k^{(0)}\|_{L_y^2}\\
	 &\qquad\qquad\qquad\quad+(\nu^{-1}|k|)^{\frac{N}3}\big(\|\omega_k^{(0)}\|_{L_y^2}+(\nu\mu)^{-\frac16}|k|^\frac13\|\theta_k^{(0)}\|_{L_y^2}\big)\Big).
	\end{split}
	\end{equation}
\end{prop}

\vskip .1in
A similar linear stability result for a slightly different domain was obtained in \cite{Tao}, but the proof presented here is different, simpler and  more compact. The estimates in Proposition
\ref{linst} can be converted into a more elegant statement that allows a direct comparison with
the nonlinear stability results to be presented. We explain and define a few notations.
(\ref{rep}) clearly reveals the distinction between
the zero mode case $k=0$ and the nonzero modes $k\not =0$. This triggers the definitions
\beq\label{ssp}
f_0(y):=\frac{1}{2\pi}\int_{\mathbb T}f(x,y) dx, \qquad f_{\not =}(x,y) =f(x,y) -f_0(y),
\eeq
which represents the projection onto $0$ frequency and the projection onto non-zero frequencies.
In the process of deriving (\ref{rep}), we made the change of variable $\eta = \xi + t k$, which
naturally invites the definition of the time-dependent elliptic operator, for $t\ge 0$,
\beq \label{elli}
\Lambda_t^2=1-\p_x^2-(\p_y+t\p_x)^2,
\eeq
or, in terms of its symbol, $\Lambda_t^2 (k, \xi) = 1 + k^2 + (\xi + t k)^2$. It is easy to check that the operator $\Lambda_t^b$ with any $b\in \mathbb R$ commutes with the differential operator with variable coefficients $\p_t + y\p_x$, namely
$$
\Lambda_t^b\, (\p_t + y\p_x) = (\p_t + y\p_x) \, \Lambda_t^b.
$$
Therefore, applying $\Lambda_t^b$ allows us to obtain the derivative estimates without destroying
the structure of the linearized equation \eqref{both}.  Furthermore, $\Lambda^b_t$ shares similarities
with the standard fractional Laplacian operators. For example, for any $b>0$,
$$
\|\Lambda_t^b(fg)\|_{L^2}\leq \|f\|_{L^\infty}\|\Lambda_t^bg\|_{L^2}+\|g\|_{L^\infty}\|\Lambda_t^bf\|_{L^2}
$$
and, for $b>1$,
$$
\|f(t)\|_{L^\infty(\Omega)}\leq C\|\widehat f(t)\|_{L^1(\Omega)}\leq C\|\Lambda_t^bf(t)\|_{L^2(\Omega)}.
$$
To precisely state the second linear stability result,  we define the horizontal fractional
derivative by
$$
\widehat{|D_x|^\gamma f}(k, \xi) = |k|^{\gamma} \widehat{f}(k, \xi).
$$
The linear stability result in Proposition \ref{linst} can be converted into an estimate
in the physical space.

\begin{prop} \label{linst2}
	Let $(\omega,\theta)$ be the solution to \eqref{both} with initial data $(\omega^{(0)},\theta^{(0)})$. Then there exists $C>0$ such that 
	for $b\in\R$,
	\begin{equation*}
	\begin{split}
	&\|\Lambda_t^b\omega\|_{L_t^\infty(L^2)}+\nu^\frac12\|D_y\Lambda_t^b\omega\|_{L_t^2(L^2)}
	+ \sigma \nu^\frac12 \|D_x \Lambda_t^b\omega\|_{L_t^2(L^2)} +\nu^\frac16\||D_x|^{\frac13}\Lambda_t^b\omega\|_{L_t^2(L^2)}\\
	&+(\nu\mu)^{-\frac16}\Big(\||D_x|^\frac13\Lambda_t^b\theta\|_{L_t^\infty(L^2)}
	+\mu^\frac12\|D_y|D_x|^\frac13\Lambda_t^b\theta\|_{L_t^2(L^2)} + \sigma  \mu^{\frac12}\||D_x|^\frac43\Lambda_t^b\theta\|_{L_t^2(L^2)}\\
	&\qquad\qquad\quad +\mu^\frac16\||D_x|^{\frac23}\Lambda_t^b\theta\|_{L_t^2(L^2)}\Big)\\
	&\leq C\big(\|\omega^{(0)}\|_{H^b}+(\nu\mu)^{-\frac16}\||D_x|^{\frac13}\theta^{(0)}\|_{H^b}\big).
	\end{split}
	\end{equation*}
\end{prop}

\vskip .1in

We assume $\nu=\mu$ for simplicity from now on.
The main focus of this paper is actually the nonlinear stability.
We are able to establish
the stability and large-time behavior for both the full dissipation case and the case with
only vertical dissipation. Certainly the proof for the vertical dissipation case also works
for the full dissipation. Both results are presented here for a direct comparison. When the
dissipation is degenerate, more strict assumptions have to be made on the initial data.
The stability result for the fully dissipative Boussinesq equation is stated in the
following theorem.

\begin{thm}\label{Non1}
	Assume $b>1$, $\beta\geq\frac12$, $\delta\geq\beta+\frac13$, $\alpha\geq\delta-\beta+\frac23$ and that the initial data $(\omega^{(0)},\theta^{(0)})$ satisfies
	$$\|\omega^{(0)}\|_{H^b}\leq \varepsilon \nu^\beta,\qquad \|\theta^{(0)}\|_{H^b}\leq \varepsilon\nu^{\alpha},\qquad\||D_x|^\frac13\theta^{(0)}\|_{H^b}\leq \varepsilon\nu^\delta,$$
	for some sufficiently small $\varepsilon>0$.
	Then the solution $(\omega,\theta)$ to \eqref{S1eq.p} satisfies that
    \beno
	&&\|\Lambda_t^b\omega\|_{L_t^\infty(L^2)}+\nu^\frac12\|\nabla\Lambda_t^b\omega\|_{L_t^2(L^2)}
	+\nu^\frac16\| |D_x|^\frac13\Lambda_t^b\omega\|_{L_t^2(L^2)}\\
	&& \qquad\qquad\qquad \qquad \qquad\qquad +\|(-\Delta)^{-\frac12}\Lambda_t^b\omega_{\neq}\|_{L_t^2(L^2)}\leq C\varepsilon\nu^\beta,
    \\
	&&\|\Lambda_t^b\theta\|_{L_t^\infty(L^2)} +\nu^\frac12\|\nabla\Lambda_t^b \theta\|_{L_t^2(L^2)}
	+\nu^\frac16\||D_x|^\frac13\Lambda_t^b\theta\|_{L_t^2(L^2)}\\
	&& \qquad\qquad\qquad \qquad \qquad\qquad
	+\|(-\Delta)^{-\frac12}\Lambda_t^b\theta_{\neq}\|_{L_t^2(L^2)}
	\leq  C\varepsilon\nu^\alpha
	\eeno
	and
    \beno
	&& \||D_x|^\frac13\Lambda_t^b\theta\|_{L_t^\infty(L^2)}+ \nu^\frac12\|\nabla|D_x|^\frac13\Lambda_t^b \theta\|_{L_t^2(L^2)}+\nu^\frac16\||D_x|^\frac23\Lambda_t^b\theta\|_{L_t^2(L^2)}\\
	&&\qquad\quad \qquad\qquad\qquad \qquad\qquad  +\|(-\Delta)^{-\frac12}|D_x|^\frac13\Lambda_t^b\theta_{\neq}\|_{L_t^2(L^2)}
	\leq C\varepsilon\nu^\delta.
	 \eeno
\end{thm}

\vskip .1in
In the case when there is only vertical dissipation, the stability and large-time behavior result is stated as follows.

\begin{thm}\label{Non2}
	Let $b>\frac43$, $\beta\geq\frac23$, $\delta\geq\beta+\frac13$, $\alpha\geq\delta-\beta+\frac23$. Assume that
	$$\|\omega^{(0)}\|_{H^b}\leq \varepsilon\nu^\beta, \qquad\|\theta^{(0)}\|_{H^b}\leq\varepsilon\nu^\alpha, \qquad\||D_x|^\frac13\theta^{(0)}\|_{H^b}\leq\varepsilon\nu^\delta$$
	for some sufficiently small $\varepsilon>0$. Then the solution to the system \eqref{vt} with initial data $(\omega^{(0)},\theta^{(0)})$ satisfies
	\beno
	&&\|\Lambda_t^b\omega\|_{L_t^\infty(L^2)}+\nu^\frac12\|D_y\Lambda_t^b\omega\|_{L_t^2(L^2)}
	+\nu^\frac16\| |D_x|^\frac13\Lambda_t^b\omega\|_{L_t^2(L^2)}\\
	&&\qquad\qquad\qquad\qquad\qquad\qquad +\|(-\Delta)^{-\frac12}\Lambda_t^b\omega_{\neq}\|_{L_t^2(L^2)}\leq C\varepsilon\nu^\beta, \\
	&&\|\Lambda_t^b\theta\|_{L_t^\infty(L^2)}+\nu^\frac12\|D_y\Lambda_t^b \theta\|_{L_t^2(L^2)}
	+\nu^\frac16\||D_x|^\frac13\Lambda_t^b\theta\|_{L_t^2(L^2)} \\
	&&\qquad\qquad\qquad\qquad\qquad\qquad
	+\|(-\Delta)^{-\frac12}\Lambda_t^b\theta_{\neq}\|_{L_t^2(L^2)}\leq C\varepsilon\nu^\alpha
	\eeno
	and
	\beno
	&& \||D_x|^\frac13\Lambda_t^b\theta\|_{L_t^\infty(L^2)}+\nu^\frac12\|D_y|D_x|^\frac13\Lambda_t^b \theta\|_{L_t^2(L^2)}+\nu^\frac16\||D_x|^\frac23\Lambda_t^b\theta\|_{L_t^2(L^2)}\\
	&&\qquad\qquad\qquad \qquad \qquad\qquad  +\|(-\Delta)^{-\frac12}|D_x|^\frac13\Lambda_t^b\theta_{\neq}\|_{L_t^2(L^2)}\leq C\varepsilon\nu^\delta.
	\eeno
\end{thm}

\vskip .1in
Special consequences of Theorem \ref{Non1} and Theorem \ref{Non2} are the nonlinear stability for
the 2D Navier-Stokes equation with full dissipation or with only vertical dissipation.
When $\theta\equiv0$, the system \eqref{S1eq.p} reduces to the 2D Navier-Stokes vorticity equation with full dissipation.
The stability problem of the 2D Couette flow or more general shear flows near the Couette flow has previously been investigated on the 2D Navier-Stokes equations with full dissipation, we refer to the references \cite{Bed4, Mas0, Mas1}. In particular, we recover the threshold index estimate $\beta\geq\frac12$ with data in $H^b$, $b>1$ established firstly in \cite{Bed4}.
On the other hand, since the stability result for the 2D Navier-Stokes equation with only vertical dissipation is
completely new, we state it as a corollary. When $\theta \equiv 0$,  the system \eqref{vt} reduces to the 2D Navier-Stokes vorticity equation with only vertical dissipation,
\begin{equation}\label{vt1}
\begin{cases}
\p_t \omega+y\p_x \omega+({\bf u}\cdot\nabla)  \omega =\nu \p_{yy} \omega,\\
{\bf u}=-\nabla^\perp(-\Delta )^{-1}\omega.
\end{cases}
\end{equation}
Theorem \ref{Non2} yields the following stability result for (\ref{vt1}).
\begin{coro} \label{Nas}
	Let $b>\frac43$ and $\beta>\frac23$. Assume the initial vorticity $\omega^{(0)}$ satisfies
	$$\|\omega^{(0)}\|_{H^b}\leq \varepsilon \nu^\beta$$
	for some suitable small number $\varepsilon>0$. Then the corresponding solution $\om$ to
	(\ref{vt1}) satisfies
	 \beno
	&&\|\Lambda_t^b\omega\|_{L_t^\infty(L^2)}+\nu^\frac12\|D_y\Lambda_t^b\omega\|_{L_t^2(L^2)}
	+\nu^\frac16\| |D_x|^\frac13\Lambda_t^b\omega\|_{L_t^2(L^2)}\\
	&& \qquad\qquad\qquad \qquad \qquad\qquad +\|(-\Delta)^{-\frac12}\Lambda_t^b\omega_{\neq}\|_{L_t^2(L^2)}\leq C\varepsilon\nu^\beta.
	\eeno
\end{coro}

\begin{rem}
When we consider shear flows ${\bf u}_{sh}=(u(y),0)$ different from the Couette flow, the corresponding perturbation systems contain nonlocal terms which will bring extra technical difficulties.
Stability problem of more general shear flows close to the Couette flow for the Boussinesq system will be investigated in a forthcoming paper.
\end{rem}

 \subsection{Sketch of the proof}

\vskip .1in
The proofs of the nonlinear stability results stated in Theorem \ref{Non1} and Theorem \ref{Non2}
are not trivial. As aforementioned, due to the presence of the buoyancy force, it is not plausible
to establish the desired stability results without taking full advantage of the enhanced
dissipation, created by the combination of $y\p_x \om$ with $\p_{yy} \om$ in the vorticity equation and of $y\p_x \theta$ with $\p_{yy} \theta$ in the temperature equation.

\vskip .1in
Let us explain how to extract the enhanced dissipation, especially the regularity in the horizontal direction, generated by the non-self-adjoint operator $y\p_x-\nu\p_{yy}$. We design a Fourier multiplier operator $\mathcal M$ defined as follows.
Choose a real-valued, non-decreasing function $\varphi\in C^\infty(\R)$ satisfying $0\leq\varphi\leq1$ and $\varphi'=\frac14$ on $[-1,1]$. Define the multiplier ${\mathcal M}={\mathcal M}(D_x,D_y)$ as ${\mathcal M}={\mathcal M}_1+{\mathcal M}_2+1$ with ${\mathcal M}_1$ and ${\mathcal M}_2$ given by
\begin{equation*}\label{multiplier}
\begin{split}
{\mathcal M}_1(k,\xi)&=\varphi\big(\nu^\frac13|k|^{-\frac13}{\rm sgn}(k)\xi\big),\  k\neq0,\\
{\mathcal M}_2(k,\xi)&=\frac{1}{ k^2}\Big(\arctan\frac{\xi}{k}+\frac{\pi}{2}\Big),\ k\neq0, \\
{\mathcal M}_1(0,\xi)&={\mathcal M}_2(0,\xi)=0.
\end{split}
\end{equation*}
Then ${\mathcal M}$ is a self-adjoint Fourier multiplier acting on $L^2(\Omega)$ and verifies that
$$1\leq{\mathcal M}\leq2+\pi.$$
Let us remark the fact that for a self-adjoint operator $A=A^\ast$ and a skew-adjoint operator $B=-B^\ast$ on $L^2$, we have the following identity
\begin{equation}\label{idcom}
\begin{split}
2{\rm Re}\langle Af,Bf\rangle_{L^2}&=\langle Af,Bf\rangle_{L^2}+\langle Bf,Af\rangle_{L^2}\\
&=\langle B^\ast Af,f\rangle_{L^2}+\langle A^\ast Bf,f\rangle_{L^2}\\
&=\langle(AB-BA)f,f\rangle_{L^2}=\langle[A,B]f,f\rangle_{L^2},
\end{split}
\end{equation}
where $[A,B]:=AB-BA$ denotes the commutator between $A$ and $B$.

%
Now taking the inner product of $(y\p_x-\nu\p_{yy})\omega$ with ${\mathcal M}\omega$ leads to the quantity
$$
R:= 2{\rm Re}\langle y\p_x \om, \mathcal M \om\rangle_{L^2} - 2\nu {\rm Re}\langle \p_{yy} \om, \mathcal M \om\rangle_{L^2},
$$
for which we intend to prove a lower bound.
Using the fact that $\mathcal M$ is self-adjoint and $y\p_x$ is skew-adjoint, we have
\ben
2{\rm Re}\langle y\p_x  \om, \mathcal M \om\rangle_{L^2}
&=&  \langle[\mathcal M, y \p_x]   \om,  \om\rangle_{L^2} \notag \\ 
&=& \sum_{k} \int_{\mathbb R} ( k \p_\xi \mathcal M) \,|\widehat{ \om}(k, \xi)|^2
\, d\xi, \notag
\een
where we have used Plancherel's theorem in the last step. Consequently,
$$
R =  \sum_{k} \int_{\mathbb R} \left(k \p_\xi \mathcal M
+ 2 \nu\,\mathcal M \xi^2\right) \,|\widehat{\om}(k, \xi)|^2\, d\xi.
$$
The multiplier ${\mathcal M}_1$ is constructed in order to capture the regularity in the horizontal direction:
according to the definition of $\mathcal M_1$, for any $k\neq0$ and $\xi\in\R$,
$$
k\p_\xi{\mathcal M}_1(k,\xi)=\nu^{\frac13}|k|^\frac23\varphi'\big(\nu^\frac13|k|^{-\frac13}{\rm sgn}(k)\xi\big),
$$
which is bounded from below by $\frac14\nu^\frac13|k|^\frac23$ when $|\xi|\leq\nu^{-\frac13}|k|^\frac13$, thanks to the special choice of the function $\varphi$. One finds the following important inequality
\begin{equation*}\label{}
 \nu \xi^2  +  k \p_\xi \mathcal M_1  \ge  \frac14 \nu^{\frac13} |k|^{\frac23} ,\quad \forall \xi\in\R,\  k\in\Z.
\end{equation*}
The multiplier ${\mathcal M}_2$ is designed to control the velocity in the nonlinear term since we have
$$k\p_\xi{\mathcal M}_2(k,\xi)=\frac1{k^2+\xi^2}.$$
Combining the above estimates, one achieves the lower bound
\beq \label{enhan}
R \ge  \nu \|\p_y  \om\|_{L^2}^2 + \frac14 \nu^{\frac13} \||D_x|^{\frac13} \om\|_{L^2}^2 + \|(-\Delta)^{\frac12}  \om_{\not =}\|_{L^2}^2.
\eeq
\eqref{enhan} leads to a control of $\frac13$-horizontal
derivative of $\om$ and this is the main reason why we can possibly control the buoyancy force, as well as the nonlinear terms.
Let us also remark that the exponent $1/3$ on the right hand side of \eqref{enhan} is sharp in the sense that
 there exist $c>0$ and functions $\omega_\nu\in L^2$ such that the equality $\|(y\p_x-\nu\p_{yy})\omega_{\nu}\|_{L^2}\|\omega_{\nu}\|_{L^2}=c\nu^\frac13\||D_x|^\frac13\omega_\nu\|_{L^2}^2$ holds for all $0<\nu<1$.
 This is due to the special first-order bracket structure of the operator $y\p_x-\nu\p_{yy}$, see \cite{Ho} for more details.

 Standard Sobolev type energy estimates would not work since they would destroy the combination, see Proposition \ref{linst}.
 We shall apply the operator $\Lambda_t^b$ defined in \eqref{elli} which allows differentiate the equations in \eqref{S1eq.p} and \eqref{vt} without changing the linear structures of the system, and then apply the multiplier ${\mathcal M}$ to obtain the desired enhanced dissipations for higher-order derivatives.

\vskip .1in
The buoyancy term in the equation of the vorticity $\om$ takes the form $\p_x \theta$, which
contains full one horizontal derivative. In the process of estimating $\|\Lambda^b_t \om\|_{L^2}$, the buoyancy term can be bounded by
$$
|\langle\p_x\Lambda_t^b\theta,{\mathcal M}\Lambda_t^b\omega\rangle_{L^2}|
\leq \||D_x|^\frac23\Lambda_t^b\theta\|_{L^2}\||D_x|^\frac13\Lambda_t^b\omega\|_{L^2},
$$
which contains $\frac23$-horizontal derivative on $\theta$.
Since the enhanced dissipation in the estimate of $\|\Lambda_t^b\theta\|_{L^2}$ contains only
$\frac13$-horizontal derivative dissipation, we need to estimate $\||D_x|^\frac13\Lambda_t^b\theta\|_{L^2}$ in order to control the buoyancy term.
This explains why we combine the estimates of $\| \Lambda^b_t \om\|_{L^2}$,  $\| \Lambda^b_t \theta\|_{L^2}$ and $\| |D_x|^\frac13\Lambda^b_t \om\|_{L^2}$.

\vskip .1in
Most of the efforts are devoted to
obtaining suitable upper bounds on the nonlinear terms. This is a very delicate process
especially when there is only vertical dissipation. Let us explain some of the difficulties and our approach in dealing with them when we estimate the nonlinear term ${\bf u}\cdot\nabla\theta$. The velocity $\bf u$ is represented in terms of $\om$ via the Biot-Savart law
$${\bf u}=-\nabla^\perp(-\Delta)^{-1}\om.$$
To distinguish between the different behaviors of the zeroth mode and the nonzero modes, we
split the velocity into two parts according to \eqref{ssp}
$$
{\bf u} = {\bf u}_0 + {\bf u}_{\neq} = \begin{pmatrix}u_{0}\\0\end{pmatrix} + \begin{pmatrix}u_{\neq}\\ v_{\neq} \end{pmatrix} = \begin{pmatrix}u_{0}\\0\end{pmatrix} +  \begin{pmatrix}\p_y(-\Delta)^{-1}\omega_{\neq}\\ -\p_x(-\Delta)^{-1}\omega_{\neq} \end{pmatrix},
$$
where $u_0=\p_y(-\p_y^2)^{-1}\omega_0$. Accordingly,  ${\bf u}\cdot\nabla\theta$ is decomposed into three parts,
$$
{\bf u}\cdot\nabla\theta = u_0\p_x\theta
+\p_y(-\Delta)^{-1}\omega_{\neq}\p_x\theta
-\p_x(-\Delta)^{-1}\omega_{\neq}\p_y\theta.
$$
Due to the lack of dissipation in the horizontal direction, it is impossible to obtain suitable bounds for the first two terms in ${\bf u}\cdot\nabla\theta$ directly.
Our strategy to overcome this difficulty is to estimate the scalar product
$$H:=\langle\Lambda_t^b( {\bf u}\cdot\nabla\theta),{\mathcal M}\Lambda_t^b\theta\rangle_{L^2}.$$
With the help of the multiplier ${\mathcal M}$, the frequency space is divided into different subdomains to facilitate cancellations and derivative distributions. Commutator estimates are employed to shift derivatives so that we are  able to control the nonlinear terms. Detailed estimates are very technical and left to the proof of Theorem \ref{Non2} in Section \ref{nonpr2}.

\vskip .1in
The rest of this paper is divided into three sections. Section \ref{linpr} proves the linear stability
stated in Propositions \ref{linst} and \ref{linst2}. Theorem
\ref{Non1} is proved in Section \ref{nonpr1} while Section \ref{nonpr2} presents the proof of Theorem \ref{Non2}.

\vskip .3in
\section{Proofs of Propositions \ref{linst} and \ref{linst2}}
\label{linpr}

This section is devoted to the proofs of the linear stability results stated in
Propositions \ref{linst} and \ref{linst2}. These results are valid for both the full dissipation case and the case with only vertical dissipation. To prove the desired stability results,
we construct special Fourier multiplier operators to extract the enhanced dissipation
from the non-self-adjoint operators $y\p_x-\nu\p_{yy}$ and $y\p_x-\mu\p_{yy}$.

\vskip .1in
We are ready to prove Proposition \ref{linst}.

\begin{proof}[Proof of Proposition \ref{linst}] By projecting the equations in (\ref{both}) onto each frequency, we obtain the system in the $y$-variable only,
	\begin{equation}\label{S1.1eq2}
	\begin{cases}
	\p_t\omega_k+\nu (D_y^2+ \sigma k^2) \omega_k+iky \omega_k=ik \theta_k ,\\
	\p_t\theta_k+\mu(D_y^2+  \sigma k^2)\theta_k+iky \theta_k=0,\\
	\omega_k|_{t=0}=\omega^{(0)}_k,\quad \theta_k|_{t=0}=\theta^{(0)}_k,
	\end{cases}
	\end{equation}
where we have used the notation $D=\frac1i\p$. We note that $\sigma=1$ corresponds to the full dissipation case while $\sigma=0$ to the case with only vertical dissipation. Since $\om_k$ and $\theta_k$ may be complex-valued, the the $L_y^2$-inner product is given by
$$
\langle f, g\rangle_{L_y^2} = \int_{\mathbb R} f(y)\,\bar{g}(y)\,dy.
$$
By taking the $L_y^2$-inner product of $\theta_k$ with the second equation in \eqref{S1.1eq2},
we have
\begin{equation}\label{S1.1eq3}
\frac12\frac{d}{dt}\|\theta_k\|_{L_y^2}^2+\mu\|D_y\theta_k\|_{L_y^2}^2+\sigma \mu k^2\|\theta_k\|_{L_y^2}^2=0.
\end{equation}
To further the estimates, we define and apply Fourier multiplier operators. If $k>0$, we define a multiplier $M_k$ by
$$
M_k\theta_k:=\varphi(\mu^\frac13|k|^{-\frac13}D_y)\theta_k,
$$
where $\varphi$ is a real-valued, non-decreasing function,  $\varphi\in C^\infty(\R)$ satisfying $0\leq\varphi\leq1$ and $\varphi'=\frac14$ on $[-1,1]$. Clearly, $M_k$ is a self-adjoint and non-negative Fourier multiplier operator.
We take the $L_y^2$-inner product of the second equation in \eqref{S1.1eq2} with $M_k\theta_k$.  The following basic identities hold,
\beno
&& 2{\rm Re}\langle \p_t\theta_k,M_k\theta_k\rangle_{L_y^2}=\frac{d}{dt}\langle M_k\theta_k,\theta_k\rangle_{L_y^2},\\
&& 2{\rm Re}\langle \mu(D_y^2+ \sigma k^2)\theta_k,M_k\theta_k\rangle_{L_y^2}=\langle  2\mu(D_y^2+ \sigma k^2)M_k\theta_k,\theta_k\rangle_{L_y^2},\\
&& 2{\rm Re}\langle iky\theta_k,M_k\theta_k\rangle_{L_y^2}=\langle\big[M_k,iky\big]\theta_k ,\theta_k\rangle_{L_y^2},
\eeno
where in the last equation we have used the fact that $M_k$ is self-adjoint and $iky$ is skew-adjoint. Here the bracket in $\big[M_k,iky\big]$ denotes the standard commutator.
Noticing that
$$[M_k,iky]=\big[\varphi(\mu^\frac13|k|^{-\frac13}D_y),iky\big]=\mu^\frac13|k|^\frac23\varphi'(\mu^\frac13|k|^{-\frac13}D_y),$$
we obtain
\beno
&& \frac{d}{dt}\langle M_k\theta_k,\theta_k\rangle_{L_y^2}
+\langle 2\mu(D_y^2+ \sigma k^2)\varphi(\mu^\frac13|k|^{-\frac13}D_y) \theta_k, \theta_k \rangle_{L_y^2} \\
&&\qquad\qquad +\langle\mu^\frac13|k|^\frac23\varphi'(\mu^\frac13|k|^{-\frac13}D_y)\theta_k,
\theta_k\rangle_{L_y^2}=0.
\eeno
Together with \eqref{S1.1eq3}, this gives
\beno
&&\frac{d}{dt}\Big(\|\theta_k\|_{L_y^2}^2+\langle M_k\theta_k,\theta_k\rangle_{L_y^2}\Big)\\
&& + \Big\langle  \Big(2\mu(D_y^2+ \sigma k^2)\big(1+\varphi(\mu^\frac13|k|^{-\frac13}D_y)\big)+\mu^\frac13|k|^\frac23\varphi'(\mu^\frac13|k|^{-\frac13}D_y)\Big)\theta_k,\theta_k\Big\rangle_{L_y^2}=0.
\eeno
By the choice of the function $\varphi$, there holds
$$\mu(\xi^2+ \sigma k^2)\big(1+2\varphi(\mu^\frac13|k|^{-\frac13}\xi)\big)+\mu^\frac13|k|^\frac23\varphi'(\mu^\frac13|k|^{-\frac13}\xi)\geq\frac14\mu^\frac13|k|^\frac23$$
for all $k>0$, $\mu>0$, $\xi\in\R$.
Therefore,
\begin{equation*}
\begin{split}
&\frac{d}{dt}\big(\langle (1+M_k)\theta_k,\theta_k\rangle_{L_y^2}\big)+\mu\|D_y\theta_k\|_{L_y^2}^2+ \sigma \mu k^2\|\theta_k\|_{L_y^2}^2+\frac14\mu^\frac13|k|^\frac23\|\theta_k\|_{L_y^2}^2\leq0.
\end{split}
\end{equation*}
Integrating in $t$ and using properties of $M_k$, we obtain the first inequality in \eqref{S1.1prop1eq1} for $k>0$.
In the case when $k<0$, we define the multiplier $M_k$  by
$$M_k\theta_k:=\varphi(-\mu^\frac13|k|^{-\frac13}D_y)\theta_k,$$
and define $M_0=0$, we can deduce the first inequality in \eqref{S1.1prop1eq1} for $k\le 0$.

\vskip .1in
We prove the first inequality in \eqref{S1.1prop1eq2} by induction.
Differentiating the second equation in \eqref{S1.1eq2} with respect to $y$ leads to
$$
\p_t D_y^N\theta_k+\mu(D_y^2+ \sigma k^2)D_y^N\theta_k+ikyD_y^N\theta_k+kND_y^{N-1}\theta_k=0.
$$
Taking the $L^2_y$-inner product with $(1+M_k)D_y^N\theta_k$ then gives
\begin{equation*}
\begin{split}
&\frac{d}{dt}\langle (1+M_k)D_y^N\theta_k,D_y^N\theta_k\rangle_{L_y^2} +\mu\|D_y^{N+1}\theta_k\|_{L_y^2}^2+ \sigma \mu k^2\|D_y^{N}\theta_k\|_{L_y^2}^2\\
&\qquad+\frac14\mu^\frac13|k|^\frac23\|D_y^N\theta_k\|_{L_y^2}^2\leq-2{\rm Re}\langle kND_y^{N-1}\theta_k,(1+M_k)D_y^N\theta_k\rangle_{L_y^2}\\
&\quad\quad\quad \leq \frac18\mu^\frac13|k|^\frac23\|D_y^N\theta_k\|_{L_y^2}^2+32N^2\mu^{-\frac13}|k|^\frac43\|D_y^{N-1}\theta_k\|_{L_y^2}^2.
\end{split}
\end{equation*}
Integrating in $t$ yields
\beno
\|D_y^N\theta_k(t)\|_{L_y^2}^2 &\leq& 2\|D_y^N\theta_k^{(0)}\|_{L_y^2}^2e^{-\frac1{16}\mu^\frac13|k|^\frac23 t}\\
&&  + \,C_N\int_0^t\mu^{-\frac13}|k|^\frac43\|D_y^{N-1}\theta_k(s)\|_{L_y^2}^2e^{-\frac1{16}\mu^\frac13|k|^\frac23 (t-s)}ds.
\eeno
Then the first inequality in \eqref{S1.1prop1eq2} follows from the induction assumption. We define
$$
M_k'=\varphi(\nu^\frac13|k|^{-\frac13}{\rm sgn}(k)D_y)\text{ for } k\neq0,\qquad M_0'=0
$$
and multiply the $\omega_k$ equation by $\omega_k$ and $M_k'\omega_k$ to obtain
\begin{equation*}
\begin{split}
\frac{d}{dt} \langle (1+M_k')\omega_k,\omega_k\rangle_{L_y^2} +\nu\|D_y\omega_k\|_{L_y^2}^2&+ \sigma \nu k^2\|\omega_k\|_{L_y^2}^2+\frac14\nu^\frac13|k|^\frac23\|\omega_k\|_{L_y^2}^2\\
&\quad \leq-2{\rm Re}\langle ik\theta_k,(1+M_k')\omega_k\rangle_{L_y^2}.
\end{split}
\end{equation*}
Applying Young's inequality to the right-hand side yields
\begin{equation*}
\begin{split}
\frac{d}{dt}\big(\langle (1+M_k')\omega_k,\omega_k\rangle_{L_y^2}\big)+\nu\|D_y\omega_k\|_{L_y^2}^2+\sigma\, \nu k^2&\|\omega_k\|_{L_y^2}^2+\frac18\nu^\frac13|k|^\frac23\|\omega_k\|_{L_y^2}^2\\
&\qquad\leq32 \nu^{-\frac13}|k|^\frac43\|\theta_k\|_{L_y^2}^2 .
\end{split}
\end{equation*}
Integrating in $t$ and using the first inequality in \eqref{S1.1prop1eq1}, we obtain
\begin{equation*}
\|\omega_k(t)\|_{L_y^2}^2\leq C\big(\|\omega_k^{(0)}\|_{L_y^2}^2+(\nu\mu)^{-\frac13}|k|^\frac23\|\theta_k^{(0)}\|_{L_y^2}^2\big)e^{-\frac{1}{16}\nu^\frac13|k|^\frac23t}.
\end{equation*}
Differentiating the equation on $\omega_k$ and using the estimates for $\theta_k$, we can deduce the second inequality in \eqref{S1.1prop1eq2}, under the assumption $\nu\lesssim \mu$. This
completes the proof of Proposition \ref{linst}.
\end{proof}

\vskip .1in
Proposition \ref{linst2} is a consequence of Proposition \ref{linst}. We recall that the
operator $\Lambda_t$ defined in (\ref{elli}) commutes with $\p_t + y\p_x$, namely, for any $b\in \mathbb R$,
$$
\Lambda_t^b\, (\p_t + y \p_x) = (\p_t + y \p_x) \,\Lambda_t^b.
$$
Therefore it commutes with the linear equation in (\ref{both}).

\vskip .1in
\begin{proof}[Proof of Proposition \ref{linst2}] For any $b\in \mathbb R$, we apply $\Lambda_t^b$
to the equations in (\ref{both}). Since $\Lambda_t^b$ commutes the equations in (\ref{both}), 	
the upper bounds in Proposition \ref{linst} and the estimates in the proof of Proposition \ref{linst} remain valid if we replace
$\om$ and $\theta$ by $\Lambda^b_t \om$ and $\Lambda_t^b \theta$, respectively, in Proposition \ref{linst}. Similarly, since any horizontal derivatives also commute with the linear
equations in (\ref{both}), $|D_x|^{\frac13} \Lambda_t^b \theta$ enjoys similar estimates
as those for $\theta$. When we take the $L^2_x$-norm, or equivalently sum over $k$ of those
estimates for $\Lambda^b_t \om$ and $|D_x|^{\frac13} \Lambda_t^b \theta$, together
with the corresponding time integral bounds, we obtain the desired estimates in Proposition
\ref{linst2}, namely
\begin{equation*}
\begin{split}
&\|\Lambda_t^b\omega\|_{L_t^\infty(L^2)}+\nu^\frac12\|D_y\Lambda_t^b\omega\|_{L_t^2(L^2)}
+ \sigma \nu^\frac12 \|D_x \Lambda_t^b\omega\|_{L_t^2(L^2)} +\nu^\frac16\||D_x|^{\frac13}\Lambda_t^b\omega\|_{L_t^2(L^2)}\\
&+(\nu\mu)^{-\frac16}\Big(\||D_x|^\frac13\Lambda_t^b\theta\|_{L_t^\infty(L^2)}
+\mu^\frac12\|D_y|D_x|^\frac13\Lambda_t^b\theta\|_{L_t^2(L^2)} + \sigma \mu^{\frac12}\||D_x|^\frac43\Lambda_t^b\theta\|_{L_t^2(L^2)}\\
&\qquad\qquad\quad +\mu^\frac16\||D_x|^{\frac23}\Lambda_t^b\theta\|_{L_t^2(L^2)}\Big)\\
&\leq C\big(\|\omega^{(0)}\|_{H^b}+(\nu\mu)^{-\frac16}\||D_x|^{\frac13}\theta^{(0)}\|_{H^b}\big).
\end{split}
\end{equation*}
The coefficient $(\nu \mu)^{-\frac16}$ helps unify the bound in terms of the initial data. This
completes the proof of Proposition \ref{linst2}.
\end{proof}

\vskip .3in
\section{Proof of Theorem \ref{Non1}}
\label{nonpr1}

\vskip .1in
This section presents the proof of Theorem \ref{Non1} stating the nonlinear stability
for (\ref{S1eq.p}). The framework is the bootstrap argument, which consists of two main steps.
The first step is to establish the {\it a priori} bounds while the second is to apply
and complete the bootstrap argument by using the {\it a priori} bounds. Main efforts are
devoted to obtaining suitable {\it a priori} bounds.  As described in the introduction, one component in achieving the bounds is to
extract the enhanced dissipation by constructing and applying suitable Fourier multipliers.
Another one is to bound the nonlinear terms suitably. To do so,  we separate the horizontal zeroth
mode from the non-zeroth modes to distinguish their different behaviors. We make use of
sharp commutator estimates.

\vskip .1in
To help prepare for the proof, we recall several notations and basic facts. We make extensive
use of the operator $\Lambda_t$ defined in (\ref{elli}). The basic properties stated in
the following lemma will be used frequently.
\begin{lem} \label{Lam}
	The operator $\Lambda_t$ defined in (\ref{elli}) satisfies the following properties
	\begin{enumerate}
		\item For any $b\in \mathbb R$, $\Lambda_t^b$ commutes with $\p_t + y\p_x$, namely
		$$
		\Lambda^b_t \, (\p_t + y\p_x) = (\p_t + y\p_x) \, \Lambda^b_t.
		$$
		
		\item For any $b>0$,
		\begin{equation*}\label{S2.1eq3}
		\|\Lambda_t^b(fg)\|_{L^2}\leq \|f\|_{L^\infty}\|\Lambda_t^bg\|_{L^2}+\|g\|_{L^\infty}\|\Lambda_t^bf\|_{L^2}.
		\end{equation*}
		Moreover, for $b>1$, we have
		\begin{equation*}\label{S2.1eq4}
		\|f(t)\|_{L^\infty}\leq C\|\widehat f(t)\|_{L^1}\leq C\|\Lambda_t^bf(t)\|_{L^2}.
		\end{equation*}
		and consequently,
		$$
		\|\Lambda_t^b(fg)\|_{L^2}\leq C\, \|\Lambda_t^bf\|_{L^2}\, \|\Lambda_t^bg\|_{L^2}.
		$$
	\end{enumerate}
\end{lem}

Recall that we assume $\mu=\nu$.
The Fourier multiplier operator $\mathcal M$ employed here is defined as follows.
We choose a real-valued, non-decreasing function $\varphi\in C^\infty(\R)$ such that $0\leq\varphi\leq1$ and $\varphi'=\frac14$ on $[-1,1]$. We define the multiplier ${\mathcal M}={\mathcal M}(D_x,D_y)$ as ${\mathcal M}={\mathcal M}_1+{\mathcal M}_2+1$ with ${\mathcal M}_1$ and ${\mathcal M}_2$ given by
\begin{equation}\label{S2.1eq1}
\begin{split}
{\mathcal M}_1(k,\xi)&=\varphi\big(\nu^\frac13|k|^{-\frac13}{\rm sgn}(k)\xi\big),\  k\neq0,\\
{\mathcal M}_2(k,\xi)&=\frac{1}{ k^2}\Big(\arctan\frac{\xi}{k}+\frac{\pi}{2}\Big),\ k\neq0,\\
{\mathcal M}_1(0,\xi)&={\mathcal M}_2(0,\xi)=0.
\end{split}
\end{equation}
Then ${\mathcal M}$ is a self-adjoint Fourier multiplier and verifies that
$$1\leq{\mathcal M}\leq2+\pi.$$
Finally we recall the projectors onto the horizontal zeroth mode and the non-zeroth modes,
\begin{equation}\label{proj}
f_0:=({\mathbb P}_0f)(y)=\frac1{2\pi}\int_{\mathbb T}f(x,y)dx,\qquad f_{\neq}:={\mathbb P}_{\neq} f=f-{\mathbb P}_0f.
\end{equation}

\vskip .1in
\begin{proof}[Proof of Theorem \ref{Non1}] Applying $\Lambda_t^b$ to \eqref{S1eq.p} and invoking the properties of
$\Lambda_t^b$ in Lemma \ref{Lam}, we have
\begin{equation*}
\begin{cases}
\p_t \Lambda_t^b\omega+y\p_x \Lambda_t^b\omega-\nu \Delta \Lambda_t^b\omega+\Lambda_t^b\big(({\bf u}\cdot\nabla)  \omega\big) =\p_x \Lambda_t^b\theta ,\\
\p_t\Lambda_t^b\theta+y\p_x \Lambda_t^b\theta-\nu\Delta\Lambda_t^b\theta+\Lambda_t^b\big(({\bf u}\cdot\nabla)  \theta\big)=0.
\end{cases}
\end{equation*}
We then multiply the equations above by ${\mathcal M}\Lambda_t^b\omega$ and ${\mathcal M}\Lambda_t^b\theta$, respectively, and integrate over $\mathbb T \times \mathbb R$. The combination $y\p_x -\nu \Delta$ creates the enhanced dissipation. As we explained in the introduction,  we do not need the full Laplacian dissipation and the vertical dissipation is sufficient.  By \eqref{idcom}, we have
\begin{equation*}
\begin{split}
&2{\rm Re} \langle y\p_x f,{\mathcal M}f\rangle_{L^2}=\langle \big[ {\mathcal M}, y\p_x\big]f,f\rangle_{L^2}=\langle(k\p_\xi{\mathcal M})(D)f,f\rangle_{L^2}
\end{split}
\end{equation*}
since ${\mathcal M}$ is self-adjoint and $y\p_x$ is skew-adjoint. Invoking the equality above, we have
\begin{equation}\label{S1.2eq7'}
\begin{split}
\frac{d}{dt}\|\sqrt{\mathcal M}\Lambda_t^b\theta\|_{L^2}^2&+2\nu\|\nabla\sqrt{\mathcal M}\Lambda_t^b \theta\|_{L^2}^2
+\langle (k\p_\xi{\mathcal M})(D)\Lambda_t^b\theta,\Lambda_t^b\theta\rangle_{L^2}\\
&\qquad\qquad\qquad
+2\langle\Lambda_t^b\big({\bf u}\cdot\nabla\theta\big),{\mathcal M}\Lambda_t^b\theta\rangle_{L^2}=0.
\end{split}
\end{equation}
Similarly,
\begin{equation}\label{S1.2eq6}
\begin{split}
&\frac{d}{dt}\|\sqrt{\mathcal M}\Lambda_t^b\omega\|_{L^2}^2+2\nu\|\nabla\sqrt{\mathcal M}\Lambda_t^b\omega\|_{L^2}^2
+\langle (k\p_\xi{\mathcal M})(D)\Lambda_t^b\omega,\Lambda_t^b\omega\rangle_{L^2}\\
&\qquad\qquad + 2\langle\Lambda_t^b\big({\bf u}\cdot\nabla\omega\big),{\mathcal M}\Lambda_t^b\omega\rangle_{L^2}
=2\langle\p_x\Lambda_t^b\theta,{\mathcal M}\Lambda_t^b\omega\rangle_{L^2}.
\end{split}
\end{equation}
Multiplying the $\theta$ equation by ${\mathcal M}|D_x|^\frac23\Lambda_t^b\theta$ gives
\begin{equation}\label{S1.2eq7}
\begin{split}
&\frac{d}{dt}\|\sqrt{\mathcal M}|D_x|^\frac13\Lambda_t^b\theta\|_{L^2}^2+2\nu\|\nabla\sqrt{\mathcal M}|D_x|^\frac13\Lambda_t^b \theta\|_{L^2}^2\\
&+\langle|D_x|^\frac23 (k\p_\xi{\mathcal M})(D)\Lambda_t^b\theta,\Lambda_t^b\theta\rangle_{L^2}
+2 \langle\Lambda_t^b\big({\bf u}\cdot\nabla\theta\big),|D_x|^\frac23{\mathcal M}\Lambda_t^b\theta\rangle_{L^2}=0,
\end{split}
\end{equation}
According to the definition \eqref{S2.1eq1} of ${\mathcal M}$, we have
\begin{equation*}
\begin{split}
k\p_\xi{\mathcal M}(k,\xi)=\nu^{\frac13}|k|^\frac23\varphi'\big(\nu^\frac13|k|^{-\frac13}{\rm sgn}(k)\xi\big)+\frac{1}{k^2+\xi^2}
\end{split}
\end{equation*}
for $k\neq0$, $\xi\in\R$.
This implies that, for $k\neq0$, $\xi\in\R$,
$$
2\nu(\xi^2+k^2){\mathcal M}(k,\xi)+k\p_\xi{\mathcal M}(k,\xi)\geq\nu(\xi^2+k^2)+\frac14\nu^\frac13|k|^\frac23+\frac1{\xi^2+k^2}.
$$
 Therefore,
\begin{equation*}
\begin{split}
2\nu\|\nabla\sqrt{\mathcal M}f\|_{L^2}^2&+\langle(k\p_\xi{\mathcal M})(D)f,f\rangle_{L^2}\\
&\geq\nu\|\nabla f\|_{L^2}^2+\frac14\nu^\frac13\||D_x|^\frac13f\|_{L^2}^2+\|(-\Delta)^{-\frac12}f_{\neq}\|_{L^2}^2,
\end{split}
\end{equation*}
where $f_{\neq}$ is defined by \eqref{proj}. \eqref{S1.2eq6}, \eqref{S1.2eq7'} and  \eqref{S1.2eq7} then becomes
\begin{equation}\label{S1.2eq9}
\begin{split}
&\frac{d}{dt}\|\sqrt{\mathcal M}\Lambda_t^b\omega\|_{L^2}^2+\nu\|\nabla\Lambda_t^b\omega\|_{L^2}^2
+\frac14\nu^\frac13\| |D_x|^\frac13\Lambda_t^b\omega\|_{L^2}^2+\|(-\Delta)^{-\frac12}\Lambda_t^b\omega_{\neq}\|_{L^2}^2\\
&\qquad\qquad\qquad \leq 2\, \underbrace{\langle\p_x\Lambda_t^b\theta,{\mathcal M}\Lambda_t^b\omega\rangle_{L^2}}_{=I_1}
-2 \underbrace{\langle\Lambda_t^b\big({\bf u}\cdot\nabla\omega\big),{\mathcal M}\Lambda_t^b\omega\rangle_{L^2}}_{=I_2},
\end{split}
\end{equation}
\begin{equation}\label{S1.2eq10}
\begin{split}
\frac{d}{dt}\|\sqrt{\mathcal M}\Lambda_t^b\theta\|_{L^2}^2&+\nu\|\nabla\Lambda_t^b \theta\|_{L^2}^2
+\frac14\nu^\frac13\||D_x|^\frac13\Lambda_t^b\theta\|_{L^2}^2+\|(-\Delta)^{-\frac12}\Lambda_t^b\theta_{\neq}\|_{L^2}^2\\
&\qquad\qquad\qquad\qquad\qquad
\leq -2\underbrace{\langle\Lambda_t^b\big({\bf u}\cdot\nabla\theta\big),{\mathcal M}\Lambda_t^b\theta\rangle_{L^2}}_{=I_3}
\end{split}
\end{equation}
and
\begin{equation}\label{S1.2eq11}
\begin{split}
&\frac{d}{dt}\|\sqrt{\mathcal M}|D_x|^\frac13\Lambda_t^b\theta\|_{L^2}^2+\nu\|\nabla|D_x|^\frac13\Lambda_t^b \theta\|_{L^2}^2+\frac14\nu^\frac13\||D_x|^\frac23\Lambda_t^b\theta\|_{L^2}^2\\
&\qquad+\|(-\Delta)^{-\frac12}|D_x|^\frac13\Lambda_t^b\theta_{\neq}\|_{L^2}^2
\leq -2\underbrace{\langle\Lambda_t^b\big({\bf u}\cdot\nabla\theta\big),|D_x|^\frac23{\mathcal M}\Lambda_t^b\theta\rangle_{L^2}}_{=I_4}.
\end{split}
\end{equation}
Using the $L^2$-boundedness of ${\mathcal M}$, we have
\ben
|I_1|=|\langle\p_x\Lambda_t^b\theta,{\mathcal M}\Lambda_t^b\omega\rangle_{L^2}|\leq \||D_x|^\frac23\Lambda_t^b\theta\|_{L^2}\||D_x|^\frac13\Lambda_t^b\omega\|_{L^2}. \label{i1b}
\een
Since ${\bf u}$ is given by $\om$ via the Biot-Savart law,
\begin{equation*}
{\bf u}=-\nabla^\perp(-\Delta)^{-1}\omega=\begin{pmatrix}\p_y(-\Delta)^{-1}\omega\\ -\p_x(-\Delta)^{-1}\omega \end{pmatrix},
\end{equation*}
we can decompose ${\bf u}$ into two parts according to \eqref{proj},
\begin{equation}\label{S1.2eq8}
\begin{split}
{\bf u}_0&={\mathbb P}_0{\bf u}=\begin{pmatrix}u_{0}\\0\end{pmatrix},\quad\text{with } u_0=\p_y(-\p_y^2)^{-1}\omega_0,\\
{\bf u}_{\neq}&={\mathbb P}_{\neq}{\bf u}=-\nabla^\perp(-\Delta)^{-1}\omega_{\neq}.
\end{split}
\end{equation}
Therefore we can write
\begin{equation*}
\begin{split}
I_2&=\langle\Lambda_t^b\big({\bf u}\cdot\nabla\omega\big),{\mathcal M}\Lambda_t^b\omega\rangle_{L^2}=I_{21}+I_{22},
\end{split}
\end{equation*}
with
\begin{equation*}
\begin{split}
I_{21}&=\langle\Lambda_t^b\big({\bf u}_{\neq}\cdot\nabla\omega\big),{\mathcal M}\Lambda_t^b\omega\rangle_{L^2},\qquad
I_{22}=\langle\Lambda_t^b\big({\bf u}_0\cdot\nabla\omega\big),{\mathcal M}\Lambda_t^b\omega\rangle_{L^2}.
\end{split}
\end{equation*}
Using the boundedness of ${\mathcal M}$ and Lemma \ref{Lam}, we have for $b>1$,
\begin{equation*}
\begin{split}
|I_{21}|&\leq \|\Lambda_t^b\big({\bf u}_{\neq}\cdot\nabla\omega\big)\|_{L^2}\|\Lambda_t^b\omega\|_{L^2}\leq\|\Lambda_t^b{\bf u}_{\neq}\|_{L^2}\|\nabla\Lambda_t^b\omega\|_{L^2}\|\Lambda_t^b\omega\|_{L^2}.
\end{split}
\end{equation*}
By \eqref{S1.2eq8},
\begin{equation*}
\|\Lambda_t^b{\bf u}_{\neq}\|_{L^2}\leq \|\nabla^\perp(-\Delta)^{-1}\Lambda_t^b\omega_{\neq}\|_{L^2}\leq\|(-\Delta)^{-\frac12}\Lambda_t^b\omega_{\neq}\|_{L^2},
\end{equation*}
Therefore,  for $b>1$,
\begin{equation*}
\begin{split}
|I_{21}|&\leq\|(-\Delta)^{-\frac12}\Lambda_t^b\omega_{\neq}\|_{L^2}\|\nabla\Lambda_t^b\omega\|_{L^2}\|\Lambda_t^b\omega\|_{L^2}.
\end{split}
\end{equation*}
The key point is to bound $I_{22}$.
To simplify the notation, we write ${\mathcal M}_t^b =\sqrt{ \mathcal M }\Lambda_t^b$ or
$$
{\mathcal M}_t^b(k,\xi):=\sqrt{\mathcal M(k,\xi)}\, \Lambda_t^b (k, \xi) =\sqrt{\mathcal M(k,\xi)}\, (1+k^2+(\xi+kt)^2)^{b/2}.
$$
It follows from \eqref{S1.2eq8} that ${\bf u}_0\cdot\nabla\omega=u_0\p_x\omega=u_0\p_x\omega_{\neq}$ since $\omega_0$ is independent of $x$.
Therefore,
\begin{equation*}
\begin{split}
I_{22}&=\langle\Lambda_t^b\big({\bf u}_0\cdot\nabla\omega\big),{\mathcal M}\Lambda_t^b\omega\rangle_{L^2}=\langle{\mathcal M}_t^b( u_0\p_x\omega_{\neq}),{\mathcal M}_t^b\omega\rangle_{L^2},
\end{split}
\end{equation*}
Due to the cancellations
$$\langle{\mathcal M}_t^b( u_0\p_x\omega_{\neq}),{\mathcal M}_t^b\omega_0\rangle_{L^2}=0,$$
$$\langle u_0\p_x({\mathcal M}_t^b\omega_{\neq}),{\mathcal M}_t^b\omega_{\neq}\rangle_{L^2}=0,$$
we have
\begin{equation*}
\begin{split}
I_{22}&=\langle{\mathcal M}_t^b( u_0\p_x\omega_{\neq}),{\mathcal M}_t^b\omega_{\neq}\rangle_{L^2}\\
&=\langle{\mathcal M}_t^b( u_0\p_x\omega_{\neq})-u_0\p_x({\mathcal M}_t^b\omega_{\neq}),{\mathcal M}_t^b\omega_{\neq}\rangle_{L^2}.
\end{split}
\end{equation*}
By Plancherel's theorem,
\begin{equation*}
\begin{split}
I_{22}
=\sum_{k\neq0}\iint \big({\mathcal M}_t^b(k,\xi)-{\mathcal M}_t^b(k,\xi-\eta)\big)\widehat u(0,\eta)ik\widehat\omega(k,\xi-\eta)
{\mathcal M}_t^b(k,\xi)\overline{\widehat \omega(k,\xi)}d\xi d\eta\\
=-\sum_{k\neq0}\iint \big({\mathcal M}_t^b(k,\xi)-{\mathcal M}_t^b(k,\xi-\eta)\big)\frac{1}{\eta}\widehat \omega(0,\eta)k\widehat\omega(k,\xi-\eta)
{\mathcal M}_t^b(k,\xi)\overline{\widehat \omega(k,\xi)}d\xi d\eta.
\end{split}
\end{equation*}
By Taylor's formula,
$$|{\mathcal M}_t^b(k,\xi)-{\mathcal M}_t^b(k,\xi-\eta)|\leq \int_0^1|\p_\xi{\mathcal M}_t^b(k,\xi-s\eta)||\eta|ds.$$
Using the explicit expression of ${\mathcal M}_t^b$ we deduce that
\begin{equation*}
\begin{split}
|\p_\xi{\mathcal M}_t^b(k,\xi)|&\leq C\big(\nu^\frac13|k|^{-\frac13}+\frac{1}{|k|}\big)\big(1+k^2+(\xi+kt)^2\big)^{\frac{b}2}.
\end{split}
\end{equation*}
Therefore,
\begin{equation*}
\begin{split}
|I_{22}|&\leq\sum_{k\neq0}C(\nu^\frac13|k|^{\frac23}+1)\iint \Big(\big(1+k^2+(\xi+kt)^2\big)^{\frac{b}2}+\big(1+k^2+(\xi-\eta+kt)^2\big)^{\frac{b}2}\Big)\\
&\qquad\qquad\qquad\qquad\quad\qquad\qquad\quad\times|\widehat \omega(0,\eta)||\widehat\omega(k,\xi-\eta) |
{\mathcal M}_t^b(k,\xi)|\widehat \omega(k,\xi)|d\xi d\eta\\
&\leq C\nu^\frac13\|\Lambda_t^b\omega_0\|_{L^2}\||D_x|^\frac13\Lambda_t^b\omega\|_{L^2}^2+C\|\Lambda_t^b\omega_0\|_{L^2}\|\Lambda_t^b\omega_{\neq}\|_{L^2}^2.
\end{split}
\end{equation*}
Consequently,
\ben
|I_{2}|&\leq& C\,\|(-\Delta)^{-\frac12}\Lambda_t^b\omega_{\neq}\|_{L^2}\|\nabla\Lambda_t^b\omega\|_{L^2}\|\Lambda_t^b\omega\|_{L^2} \notag\\
&&+ \,C\,\nu^\frac13\|\Lambda_t^b\omega_0\|_{L^2}\||D_x|^\frac13\Lambda_t^b\omega\|_{L^2}^2
+ \,C\,\|\Lambda_t^b\omega_0\|_{L^2}\|\Lambda_t^b\omega_{\neq}\|_{L^2}^2 \notag \\
&\leq& C\,\|(-\Delta)^{-\frac12}\Lambda_t^b\omega_{\neq}\|_{L^2}\|\nabla\Lambda_t^b\omega\|_{L^2}\|\Lambda_t^b\omega\|_{L^2} \notag\\
&&+ \,C\,\nu^\frac13\|\Lambda_t^b\omega_0\|_{L^2}\||D_x|^\frac13\Lambda_t^b\omega\|_{L^2}^2\notag \\
&& + \,C\,\|\Lambda_t^b\omega_0\|_{L^2}\, \|(-\Delta)^{-\frac12}\Lambda_t^b\omega_{\neq}\|_{L^2}\,\|\nabla\Lambda_t^b\omega\|_{L^2}. \label{i2b}
\een
$I_3$ can be bounded similarly as $I_2$. We write $I_3$ as $I_3=I_{31}+I_{32}$ with
\begin{equation*}
\begin{split}
I_{31}&=\langle\Lambda_t^b\big({\bf u}_{\neq}\cdot\nabla\theta\big),{\mathcal M}\Lambda_t^b\theta\rangle_{L^2},\qquad
I_{32}=\langle\Lambda_t^b\big({\bf u}_0\cdot\nabla\theta\big),{\mathcal M}\Lambda_t^b\theta\rangle_{L^2}
\end{split}
\end{equation*}
and obtain the following bound
\ben
|I_{3}|
&\leq& \|(-\Delta)^{-\frac12}\Lambda_t^b\omega_{\neq}\|_{L^2}\|\nabla\Lambda_t^b\theta\|_{L^2}\|\Lambda_t^b\theta\|_{L^2} \notag\\
&& + \,  C\nu^\frac13\|\Lambda_t^b\omega_0\|_{L^2}\||D_x|^\frac13\Lambda_t^b\theta\|_{L^2}^2+C\|\Lambda_t^b\omega_0\|_{L^2}\|\Lambda_t^b\theta_{\neq}\|_{L^2}^2 \notag\\
&\leq& \|(-\Delta)^{-\frac12}\Lambda_t^b\omega_{\neq}\|_{L^2}\|\nabla\Lambda_t^b\theta\|_{L^2}\|\Lambda_t^b\theta\|_{L^2} \notag\\
&& + \,  C\,\nu^\frac13\|\Lambda_t^b\omega_0\|_{L^2}\||D_x|^\frac13\Lambda_t^b\theta\|_{L^2}^2
\notag \\
&& + \, C\,\|\Lambda_t^b\omega_0\|_{L^2}\|(-\Delta)^{-\frac12}\Lambda_t^b\theta_{\neq}\|_{L^2}\,\|\nabla\Lambda_t^b\theta\|_{L^2}. \label{i3b}
\een
We decompose $I_4$ as $I_4=I_{41}+I_{42}$ with
\begin{equation*}
\begin{split}
I_{41}&=\langle\Lambda_t^b\big({\bf u}_{\neq}\cdot\nabla\theta\big),|D_x|^\frac23{\mathcal M}\Lambda_t^b\theta\rangle_{L^2},\qquad
I_{42}=\langle\Lambda_t^b\big({\bf u}_0\cdot\nabla\theta\big),|D_x|^\frac23{\mathcal M}\Lambda_t^b\theta\rangle_{L^2}.
\end{split}
\end{equation*}
The estimates for $I_{42}$ are the same as those for $I_{22}$,
\begin{equation*}
\begin{split}
|I_{42}|
&\leq C\nu^\frac13\|\Lambda_t^b\omega_0\|_{L^2}\||D_x|^\frac23\Lambda_t^b\theta\|_{L^2}^2+C\|\Lambda_t^b\omega_0\|_{L^2}\||D_x|^\frac13\Lambda_t^b\theta_{\neq}\|_{L^2}^2.
\end{split}
\end{equation*}
For $I_{41}$, we have
\begin{equation*}
\begin{split}
|I_{41}|&\leq\||D_x|^\frac13\Lambda_t^b\big({\bf u}_{\neq}\cdot\nabla\theta\big)\|_{L^2} \||D_x|^\frac13\Lambda_t^b\theta\|_{L^2}.
\end{split}
\end{equation*}
Furthermore,
\begin{equation*}
\begin{split}
\||D_x|^\frac13\Lambda_t^b\big({\bf u}_{\neq}\cdot\nabla\theta\big)\|_{L^2}\leq \|\Lambda_t^b{\bf u}_{\neq}\|_{L^2} \||D_x|^\frac13\Lambda_t^b\nabla\theta\|_{L^2}+\||D_x|^\frac13\Lambda_t^b{\bf u}_{\neq}\|_{L^2}\|\Lambda_t^b\nabla\theta\|_{L^2}
\end{split}
\end{equation*}
and
\begin{equation*}
\begin{split}
\|\Lambda_t^b{\bf u}_{\neq}\|_{L^2}\leq \|(-\Delta)^{-\frac12}\Lambda_t^b\omega_{\neq}\|_{L^2},\qquad\||D_x|^\frac13\Lambda_t^b{\bf u}_{\neq}\|_{L^2}\leq \||D_x|^\frac13\Lambda_t^b\omega\|_{L^2}.
\end{split}
\end{equation*}	
Therefore, we deduce that
\ben
|I_4| &\le&  C\nu^\frac13\|\Lambda_t^b\omega_0\|_{L^2}\||D_x|^\frac23\Lambda_t^b\theta\|_{L^2}^2+C\|\Lambda_t^b\omega_0\|_{L^2}\||D_x|^\frac13\Lambda_t^b\theta_{\neq}\|_{L^2}^2 \notag \\
&& + \, C\, \|(-\Delta)^{-\frac12}\Lambda_t^b\omega_{\neq}\|_{L^2}\, \||D_x|^\frac13\Lambda_t^b\nabla\theta\|_{L^2}\||D_x|^\frac13\Lambda_t^b\theta\|_{L^2} \notag\\
&& + \,C\,\||D_x|^\frac13\Lambda_t^b\omega\|_{L^2}\, \|\Lambda_t^b\nabla\theta\|_{L^2}\,\||D_x|^\frac13\Lambda_t^b\theta\|_{L^2}. \label{i4b}
\een
Inserting the upper bounds (\ref{i1b}), (\ref{i2b}), (\ref{i3b}) and (\ref{i4b})
in (\ref{S1.2eq9}), (\ref{S1.2eq10}) and (\ref{S1.2eq11}) and integrating in time, we obtain
\ben
&&\|\Lambda_t^b\omega\|_{L_t^\infty(L^2)}^2+\nu\|\nabla \Lambda_t^b\omega\|_{L_t^2(L^2)}^2
+\frac18\nu^\frac13\||D_x|^\frac13\Lambda_t^b\omega\|_{L_t^2(L^2)}^2 +\|(-\Delta)^{-\frac12}\Lambda_t^b\omega_{\neq}\|_{L_t^2(L^2)}^2 \notag\\
&&\leq 2 \|\Lambda_0^b\omega^{(0)}\|_{L^2}^2
+8\nu^{-\frac13}\||D_x|^\frac23\Lambda_t^b\theta\|_{L_t^2(L^2)}^2
+ \,C_1 \nu^{\frac13} \|\Lambda_t^b\omega\|_{L_t^\infty(L^2)}\||D_x|^\frac13\Lambda_t^b\omega\|_{L_t^2(L^2)}^2
\notag \\
&&\quad+ \,C_1 \|(-\Delta)^{-\frac12}\Lambda_t^b\omega_{\neq}\|_{L^2_tL^2}\|\nabla\Lambda_t^b\omega\|_{L^2_t(L^2)}\|\Lambda_t^b\omega\|_{L^\infty_t(L^2)},
\label{upp1}
\een
\ben
&&\|\Lambda_t^b\theta\|_{L_t^\infty(L^2)}^2+\nu\|\nabla\Lambda_t^b \theta\|_{L_t^2(L^2)}^2
+\frac14\nu^\frac13\||D_x|^\frac13\Lambda_t^b\theta\|_{L_t^2(L^2)}^2+\|(-\Delta)^{-\frac12}\Lambda_t^b\theta_{\neq}\|_{L_t^2(L^2)}^2 \notag\\
&&\leq 2 \|\Lambda_0^b\theta^{(0)}\|_{L^2}^2 + C_2\,  \|(-\Delta)^{-\frac12}\Lambda_t^b\omega_{\neq}\|_{L^2_t(L^2)} \|\nabla\Lambda_t^b\theta\|_{L^2_t(L^2)}\|\Lambda_t^b\theta\|_{L^\infty_t(L^2)}\notag\\
&&\quad +C_2 \nu^{\frac13} \|\Lambda_t^b\omega\|_{L_t^\infty(L^2)}\||D_x|^\frac13\Lambda_t^b\theta\|_{L_t^2(L^2)}^2
\notag \\
&& \quad  + C_2 \, \|\Lambda_t^b\omega\|_{L_t^\infty(L^2)} \,
\|(-\Delta)^{-\frac12}\Lambda_t^b\theta_{\neq}\|_{L^2_t(L^2)}\,
\|\nabla\Lambda_t^b\theta\|_{L^2_t(L^2)}
 \label{upp2}
\een
and
\ben
&&\||D_x|^\frac13\Lambda_t^b\theta\|_{L_t^\infty(L^2)}^2+\nu\|\nabla |D_x|^\frac13\Lambda_t^b \theta\|_{L_t^2(L^2)}^2+\frac14\nu^\frac13\||D_x|^\frac23\Lambda_t^b\theta\|_{L_t^2(L^2)}^2
\notag\\
&&\quad+\|(-\Delta)^{-\frac12}|D_x|^\frac13\Lambda_t^b\theta_{\neq}\|_{L_t^2(L^2)}^2\notag\\
&&\leq 2 \||D_x|^\frac13\Lambda_0^b\theta^{(0)}\|_{L^2}^2+C_3\nu^{\frac13} \|\Lambda_t^b\omega\|_{L_t^\infty(L^2)}\||D_x|^\frac23\Lambda_t^b\theta\|_{L_t^2(L^2)}^2 \notag\\
&&\quad+ C_3\, \|\Lambda_t^b\omega_0\|_{L^\infty_t(L^2)} \||D_x|^\frac13\Lambda_t^b\theta_{\neq}\|_{L^2_t(L^2)}^2\notag\\
&& \quad + \, C_3\, \|(-\Delta)^{-\frac12}\Lambda_t^b\omega_{\neq}\|_{L^2_t(L^2)}\, \||D_x|^\frac13\Lambda_t^b\nabla\theta\|_{L^2_t(L^2)}\||D_x|^\frac13\Lambda_t^b\theta\|_{L^\infty_t(L^2)} \notag\\
&&\quad+C_3 \||D_x|^\frac13\Lambda_t^b\omega\|_{L^2_t(L^2)}\, \|\Lambda_t^b\nabla\theta\|_{L^2_t(L^2)}\,\||D_x|^\frac13\Lambda_t^b\theta\|_{L^\infty_t(L^2)}.
\label{upp3}
\een

The {\it a priori} bounds in (\ref{upp1}), (\ref{upp2}) and (\ref{upp3}) allow us to prove Theorem \ref{Non1} through the bootstrap argument. We recall the assumptions on the initial
data $(\omega^{(0)},\theta^{(0)})$,
\ben
\|\omega^{(0)}\|_{H^b}\leq \varepsilon \nu^\beta,\qquad \|\theta^{(0)}\|_{H^b}\leq \varepsilon\nu^{\alpha},\qquad\||D_x|^\frac13\theta^{(0)}\|_{H^b}\leq \varepsilon\nu^\delta,
\label{iop}
\een
where $\varepsilon>0$ is sufficiently small and
\beq\label{index}
\beta\geq\frac12, \quad \delta\geq\beta+\frac13, \quad \alpha\geq\delta-\beta+\frac23.
\eeq
To apply the bootstrap argument, we make the ansatz that, for $T\le \infty$, the solution
of (\ref{S1eq.p}) obeys
\ben
&&\|\Lambda_t^b\omega\|_{L_T^\infty(L^2)}+\nu^\frac12\|\nabla \Lambda_t^b\omega\|_{L_T^2(L^2)}
+\nu^\frac16\| |D_x|^\frac13\Lambda_t^b\omega\|_{L_T^2(L^2)} \notag\\
&& \qquad\qquad \qquad\qquad \qquad\qquad +\|(-\Delta)^{-\frac12}\Lambda_t^b\omega_{\neq}\|_{L_T^2(L^2)}\leq C\varepsilon\nu^\beta, \label{ans1}\\
&&\|\Lambda_t^b\theta\|_{L_T^\infty(L^2)}+\nu^\frac12\|\nabla \Lambda_t^b \theta\|_{L_T^2(L^2)}
+\nu^\frac16\||D_x|^\frac13\Lambda_t^b\theta\|_{L_T^2(L^2)} \notag\\
&& \qquad\qquad\qquad\qquad \qquad\qquad  +\|(-\Delta)^{-\frac12}\Lambda_t^b\theta_{\neq}\|_{L_T^2(L^2)}\leq C\varepsilon\nu^\alpha,\label{ans2}
\\
&&\||D_x|^\frac13\Lambda_t^b\theta\|_{L_T^\infty(L^2)}+\nu^\frac12\|\nabla|D_x|^\frac13\Lambda_t^b \theta\|_{L_T^2(L^2)}+\nu^\frac16\||D_x|^\frac23\Lambda_t^b\theta\|_{L_T^2(L^2)}\notag\\
&&  \qquad\qquad\qquad\qquad \qquad\qquad  +\|(-\Delta)^{-\frac12}|D_x|^\frac13\Lambda_t^b\theta_{\neq}\|_{L_T^2(L^2)}\leq \widetilde C\varepsilon\nu^\delta.\label{ans3}
\een
We then show that (\ref{ans1}), (\ref{ans2}) and (\ref{ans3}) actually hold with $C$
replaced by $C/2$ and $\widetilde C$ by $\widetilde C/2$. In fact, if we insert
the initial condition (\ref{iop}) and the ansatz (\ref{ans1}), (\ref{ans2}) and (\ref{ans3})
in (\ref{upp1}), (\ref{upp2}) and (\ref{upp3}), we find
\beno
&& \|\Lambda_t^b\omega\|_{L_t^\infty(L^2)}^2+\nu\|\nabla \Lambda_t^b\omega\|_{L_t^2(L^2)}^2
+\frac18\nu^\frac13\||D_x|^\frac13\Lambda_t^b\omega\|_{L_t^2(L^2)}^2 +\|(-\Delta)^{-\frac12}\Lambda_t^b\omega_{\neq}\|_{L_t^2(L^2)}^2 \notag\\
&& \qquad \le 2 \varepsilon^2\nu^{2\beta}+8\widetilde C^2 \varepsilon^2\nu^{2\delta-\frac23} + C_1\,C^3\varepsilon^3(\nu^{3\beta-\frac13} + \nu^{3\beta-\frac12}),\\
&&\|\Lambda_t^b\theta\|_{L_t^\infty(L^2)}^2+\nu\|\nabla\Lambda_t^b \theta\|_{L_t^2(L^2)}^2
+\frac14\nu^\frac13\||D_x|^\frac13\Lambda_t^b\theta\|_{L_t^2(L^2)}^2+\|(-\Delta)^{-\frac12}\Lambda_t^b\theta_{\neq}\|_{L_t^2(L^2)}^2 \notag\\
&& \qquad\qquad \leq 2 \varepsilon^2\nu^{2\alpha}+C_2C^3\varepsilon^3(3\nu^{\beta+2\alpha-\frac12}+\nu^{\beta+2\alpha-\frac13}),\\
&&\||D_x|^\frac13\Lambda_t^b\theta\|_{L_t^\infty(L^2)}^2+\nu\|\nabla |D_x|^\frac13\Lambda_t^b \theta\|_{L_t^2(L^2)}^2+\frac14\nu^\frac13\||D_x|^\frac23\Lambda_t^b\theta\|_{L_t^2(L^2)}^2
\notag\\
&&\quad+\|(-\Delta)^{-\frac12}|D_x|^\frac13\Lambda_t^b\theta_{\neq}\|_{L_t^2(L^2)}^2\notag\\
&&\qquad\qquad \leq 2 \varepsilon^2\nu^{2\delta}+C_3C\widetilde C\varepsilon^3(2\widetilde C\nu^{\beta+2\delta}+ C\nu^{\beta + 2\alpha -\frac13}+ 2 C\nu^{\beta+\alpha+\delta-\frac23}).
\eeno
If we invoke \eqref{index} and  choose
$$\widetilde C\geq8,\quad C\geq32\widetilde C,\quad \varepsilon=\min\big(\frac1{128C_1C},\frac{1}{128C_2C},\frac{\widetilde C}{64C_3C}\big),$$
then the inequalities (\ref{ans1})-(\ref{ans2}) hold with $C$ replaced by $C/2$ and \eqref{ans3} holds with $\widetilde C$ replaced by $\widetilde C/2$.
This completes the proof of Theorem \ref{Non1}.
\end{proof}

\vskip .3in
\section{Proof of Theorem \ref{Non2}}
\label{nonpr2}

\vskip .1in
This section proves the nonlinear stability result stated in Theorem \ref{Non2}. We recall
that the Boussinesq system concerned here has only  vertical dissipation, namely
\begin{equation}\label{vt2}
\begin{cases}
\p_t \omega + y\p_x \omega+({\bf u}\cdot\nabla)  \omega =\nu \p_{yy} \omega+\p_x\theta ,\\
\p_t\theta+y\p_x \theta+({\bf u}\cdot\nabla) \theta=\nu\p_{yy}\theta,\\
{\bf u}=-\nabla^\perp(-\Delta )^{-1}\omega,\\
\om(x,0) = \om^{(0)}, \quad \theta(x,0) = \theta^{(0)}.
\end{cases}
\end{equation}
The proof is much more involved than the full dissipation case. The framework is still the
bootstrap argument, but it is now much more difficult to prove the desired {\it a priori}
bounds due to the lack of horizontal dissipation.
The Fourier multiplier operator is the same as that is designed for the full dissipation case,
but the nonlinear terms are now difficult to control.
Various techniques are combined to achieve suitable upper bounds. The quantities are decomposed into horizontal zeroth mode and the
non-zeroth modes to distinguish their different behaviors. Commutator estimates are employed to
shift derivatives. In addition, the frequency space is divided into different subdomains to
facilitate cancellations and derivative distribution.

\vskip .1in
\begin{proof}[Proof of Theorem \ref{Non2}]
Applying the operator $\Lambda_t^b$ to \eqref{vt2} and making use of the fact
that $\Lambda_t^b$ commutes with $\p_t + y \p_x$, we obtain
\begin{equation*}
\begin{cases}
\p_t \Lambda_t^b\omega+y\p_x \Lambda_t^b\omega-\nu \p_y^2 \Lambda_t^b\omega+\Lambda_t^b\big(({\bf u}\cdot\nabla)  \omega\big) =\p_x \Lambda_t^b\theta ,\\
\p_t\Lambda_t^b\theta+y\p_x \Lambda_t^b\theta-\nu\p_y^2\Lambda_t^b\theta+\Lambda_t^b\big(({\bf u}\cdot\nabla)  \theta\big)=0.
\end{cases}
\end{equation*}
We then take the scalar product of the equations with ${\mathcal M}\Lambda_t^b\omega$ and ${\mathcal M}\Lambda_t^b\theta$, respectively, where ${\mathcal M}$ is defined in \eqref{S2.1eq1}.
Using (\ref{idcom}), due to the fact that
${\mathcal M}$ is self-adjoint and $y\p_x$ is skew-adjoint,
\begin{equation*}
\begin{split}
&2{\rm Re} \langle y\p_x f,{\mathcal M}f\rangle_{L^2}=\langle \big[ {\mathcal M}, y\p_x\big]f,f\rangle_{L^2}=\langle(k\p_\xi{\mathcal M})(D)f,f\rangle_{L^2}.
\end{split}
\end{equation*}
Invoking this equality, we have
\begin{equation}\label{S2.2eq2}
\begin{split}
&\frac{d}{dt}\|\sqrt{\mathcal M}\Lambda_t^b\omega\|_{L^2}^2+2\nu\|D_y\sqrt{\mathcal M}\Lambda_t^b\omega\|_{L^2}^2
+\langle (k\p_\xi{\mathcal M})(D)\Lambda_t^b\omega,\Lambda_t^b\omega\rangle_{L^2}\\
&\qquad\qquad +2 \langle\Lambda_t^b\big({\bf u}\cdot\nabla\omega\big),{\mathcal M}\Lambda_t^b\omega\rangle_{L^2}
=2 \langle\p_x\Lambda_t^b\theta,{\mathcal M}\Lambda_t^b\omega\rangle_{L^2}
\end{split}
\end{equation}
and
\begin{equation}\label{S2.2eq1}
\begin{split}
\frac{d}{dt}\|\sqrt{\mathcal M}\Lambda_t^b\theta\|_{L^2}^2&+2\nu\|D_y\sqrt{\mathcal M}\Lambda_t^b \theta\|_{L^2}^2
+\langle (k\p_\xi{\mathcal M})(D)\Lambda_t^b\theta,\Lambda_t^b\theta\rangle_{L^2}\\
&\qquad\qquad\qquad
+2 \langle\Lambda_t^b\big({\bf u}\cdot\nabla\theta\big),{\mathcal M}\Lambda_t^b\theta\rangle_{L^2}=0.
\end{split}
\end{equation}
Similarly, taking the $L^2$-inner product of ${\mathcal M}|D_x|^\frac23\Lambda_t^b\theta$
with the $\theta$ equation gives
\begin{equation}\label{S2.2eq3}
\begin{split}
&\frac{d}{dt}\|\sqrt{\mathcal M}|D_x|^\frac13\Lambda_t^b\theta\|_{L^2}^2+2\nu\|D_y\sqrt{\mathcal M}|D_x|^\frac13\Lambda_t^b \theta\|_{L^2}^2\\
&+\langle|D_x|^\frac23 (k\p_\xi{\mathcal M})(D)\Lambda_t^b\theta,\Lambda_t^b\theta\rangle_{L^2}
+2\langle\Lambda_t^b\big({\bf u}\cdot\nabla\theta\big),|D_x|^\frac23{\mathcal M}\Lambda_t^b\theta\rangle_{L^2}=0.
\end{split}
\end{equation}
By the definition of ${\mathcal M}$, we have
\begin{equation*}
\begin{split}
k\p_\xi{\mathcal M}(k,\xi)=\nu^{\frac13}|k|^\frac23\varphi'\big(\nu^\frac13|k|^{-\frac13}{\rm sgn}(k)\xi\big)+\frac{1}{k^2+\xi^2},
\end{split}
\end{equation*}
for $k\neq0$, $\xi\in\R$.
Using the properties of the function $\varphi$, especially $\varphi'=\frac14$ when $\nu^\frac13|k|^{-\frac13} |\xi|\le 1$, we have, for $k\neq0$, $\xi\in\R$
\begin{equation*}
2\nu \xi^2 {\mathcal M}(k,\xi)+k\p_\xi{\mathcal M}(k,\xi)\geq\nu \xi^2+\frac14\nu^\frac13|k|^\frac23+\frac1{\xi^2+k^2}.
\end{equation*}
As a consequence,
\begin{equation}\label{S2.2eq4}
\begin{split}
2\nu\|D_y\sqrt{\mathcal M}f\|_{L^2}^2&+\langle(k\p_\xi{\mathcal M})(D)f,f\rangle_{L^2}\\
&\geq\nu\|D_y f\|_{L^2}^2+\frac14\nu^\frac13\||D_x|^\frac13f\|_{L^2}^2+\|(-\Delta)^{-\frac12}f_{\neq}\|_{L^2}^2,
\end{split}
\end{equation}
where $f_{\neq}$ is given in \eqref{proj}.
Inserting (\ref{S2.2eq4}) in \eqref{S2.2eq2}, \eqref{S2.2eq1}, \eqref{S2.2eq3} yields
\begin{equation}\label{S2.2eq5}
\begin{split}
&\frac{d}{dt}\|\sqrt{\mathcal M}\Lambda_t^b\omega\|_{L^2}^2+\nu\|D_y\Lambda_t^b\omega\|_{L^2}^2
+\frac14\nu^\frac13\| |D_x|^\frac13\Lambda_t^b\omega\|_{L^2}^2+\|(-\Delta)^{-\frac12}\Lambda_t^b\omega_{\neq}\|_{L^2}^2\\
&\qquad\qquad\qquad \leq 2 \underbrace{\langle\p_x\Lambda_t^b\theta,{\mathcal M}\Lambda_t^b\omega\rangle_{L^2}}_{={ I_1}}
-2 \underbrace{\langle\Lambda_t^b\big({\bf u}\cdot\nabla\omega\big),{\mathcal M}\Lambda_t^b\omega\rangle_{L^2}}_{={ I_2}},
\end{split}
\end{equation}
\begin{equation}\label{S2.2eq6}
\begin{split}
\frac{d}{dt}\|\sqrt{\mathcal M}\Lambda_t^b\theta\|_{L^2}^2&+\nu\|D_y\Lambda_t^b \theta\|_{L^2}^2
+\frac14\nu^\frac13\||D_x|^\frac13\Lambda_t^b\theta\|_{L^2}^2+\|(-\Delta)^{-\frac12}\Lambda_t^b\theta_{\neq}\|_{L^2}^2\\
&\qquad\qquad\qquad\qquad\qquad
\leq -2\underbrace{\langle\Lambda_t^b\big({\bf u}\cdot\nabla\theta\big),{\mathcal M}\Lambda_t^b\theta\rangle_{L^2}}_{={ I_3}}
\end{split}
\end{equation}
and
\begin{equation}\label{S2.2eq7}
\begin{split}
&\frac{d}{dt}\|\sqrt{\mathcal M}|D_x|^\frac13\Lambda_t^b\theta\|_{L^2}^2+\nu\|D_y|D_x|^\frac13\Lambda_t^b \theta\|_{L^2}^2+\frac14\nu^\frac13\||D_x|^\frac23\Lambda_t^b\theta\|_{L^2}^2\\
&\qquad+\|(-\Delta)^{-\frac12}|D_x|^\frac13\Lambda_t^b\theta_{\neq}\|_{L^2}^2
\leq -2 \underbrace{\langle\Lambda_t^b\big({\bf u}\cdot\nabla\theta\big),|D_x|^\frac23{\mathcal M}\Lambda_t^b\theta\rangle_{L^2}}_{={I_4}}.
\end{split}
\end{equation}
The term ${ I_1}$ is easy to deal with, using the $L^2$-boundedness of ${\mathcal M}$, we have
\begin{equation}\label{S2.2eq8}
\begin{split}
|{I_1}|=|\langle\p_x\Lambda_t^b\theta,{\mathcal M}\Lambda_t^b\omega\rangle_{L^2}|
&\leq \||D_x|^\frac23\Lambda_t^b\theta\|_{L^2}\||D_x|^\frac13\Lambda_t^b\omega\|_{L^2}\\
&\leq\frac1{16}\nu^\frac13\||D_x|^\frac13\Lambda_t^b\omega\|_{L^2}^2+8\nu^{-\frac13}\||D_x|^\frac23\Lambda_t^b\theta\|_{L^2}^2.
\end{split}
\end{equation}
{\bf Estimates for $I_2$ and $I_3$}.
The terms $I_2$ and $I_3$ have the same structure so that we only estimate $I_3$.
Recall that the velocity field ${\bf u}$ is given by the Biot-Savart law
\begin{equation*}
{\bf u}=-\nabla^\perp(-\Delta)^{-1}\omega=\begin{pmatrix}\p_y(-\Delta)^{-1}\omega\\ -\p_x(-\Delta)^{-1}\omega \end{pmatrix}=:\begin{pmatrix}u\\ v \end{pmatrix}.
\end{equation*}
According to \eqref{proj}, ${\bf u}$ can be decomposed into ${\bf u}_0$ and ${\bf u}_{\neq}$,
\begin{equation}\label{S2.3eq1}
\begin{split}
{\bf u}_0&={\mathbb P}_0{\bf u}=\begin{pmatrix}u_{0}\\0\end{pmatrix},\quad\text{with } u_0=\p_y(-\p_y^2)^{-1}\omega_0,\\
{\bf u}_{\neq}&={\mathbb P}_{\neq}{\bf u}=-\nabla^\perp(-\Delta)^{-1}\omega_{\neq}=\begin{pmatrix}\p_y(-\Delta)^{-1}\omega_{\neq}\\ -\p_x(-\Delta)^{-1}\omega_{\neq} \end{pmatrix}=\begin{pmatrix}u_{\neq}\\ v_{\neq} \end{pmatrix}.
\end{split}
\end{equation}
Thus,
$${\bf u}\cdot\nabla\theta=u_0\p_x\theta+\p_y(-\Delta)^{-1}\omega_{\neq}\p_x\theta-\p_x(-\Delta)^{-1}\omega_{\neq}\p_y\theta.$$
Then we can write
\begin{equation*}\label{S2.3eq2}
\begin{split}
I_3&=\langle\Lambda_t^b\big({\bf u}\cdot\nabla\theta\big),{\mathcal M}\Lambda_t^b\theta\rangle_{L^2}=I_{31}+I_{32}+I_{33},\quad\text{with}\\
I_{31}&:=-\langle\Lambda_t^b\big(\p_x(-\Delta)^{-1}\omega_{\neq}\p_y\theta\big),{\mathcal M}\Lambda_t^b\theta\rangle_{L^2},\\
I_{32}&:=\langle\Lambda_t^b\big({ u}_0\p_x\theta\big),{\mathcal M}\Lambda_t^b\theta\rangle_{L^2},\\
I_{33}&:=\langle\Lambda_t^b\big(\p_y(-\Delta)^{-1}\omega_{\neq}\p_x\theta\big),{\mathcal M}\Lambda_t^b\theta\rangle_{L^2}.
\end{split}
\end{equation*}
For the term $I_{31}$, we have
\begin{equation}\label{S2.3eq3}
\begin{split}
I_{31}&\leq\|\Lambda_t^b\big(\p_x(-\Delta)^{-1}\omega_{\neq}\p_y\theta\big)\|_{L^2}\|\Lambda_t^b\theta\|_{L^2}\\
&\leq \|(-\Delta)^{-\frac12}\Lambda_t^b\omega_{\neq}\|_{L^2}\|D_y\Lambda_t^b\theta\|_{L^2}\|\Lambda_t^b\theta\|_{L^2}.
\end{split}
\end{equation}
The estimates for $I_{32}$ and $I_{33}$ are much more elaborate since we only have $\frac13$-derivative enhanced dissipation in the $x$-direction, which is not enough to control $\p_x\theta$ directly. To simplify the notation, we set
$$
{\mathcal M}_t^b(k,\xi):=\sqrt{\mathcal M(k,\xi)}\, \Lambda_t^b(k,\xi).
$$
By \eqref{proj}, we write $\theta =\theta_0 + \theta_{\not =}$.  Since $\theta_0$ is independent of $x$, we have $\p_x\theta_0=0$ and the cancellations
$$\langle{\mathcal M}_t^b( u_0\p_x\theta_{\neq}),{\mathcal M}_t^b\theta_0\rangle_{L^2}=0,\qquad
\langle u_0\p_x({\mathcal M}_t^b\theta_{\neq}),{\mathcal M}_t^b\theta_{\neq}\rangle_{L^2}=0.$$
Therefore,
\begin{equation*}
\begin{split}
I_{32}&=\langle{\mathcal M}_t^b( u_0\p_x\theta_{\neq}),{\mathcal M}_t^b\theta_{\neq}\rangle_{L^2}\\
&=\langle{\mathcal M}_t^b( u_0\p_x\theta_{\neq})-u_0\p_x({\mathcal M}_t^b\theta_{\neq}),{\mathcal M}_t^b\theta_{\neq}\rangle_{L^2}.
\end{split}
\end{equation*}
Using Plancherel's theorem, we have
\begin{equation*}
\begin{split}
I_{32}
=\sum_{k\neq0}\iint \big({\mathcal M}_t^b(k,\xi)-{\mathcal M}_t^b(k,\xi-\eta)\big)\widehat u(0,\eta)ik\widehat\theta_{\neq}(k,\xi-\eta)
{\mathcal M}_t^b(k,\xi)\overline{\widehat \theta_{\neq}(k,\xi)}d\xi d\eta\\
=-\sum_{k\neq0}\iint \big({\mathcal M}_t^b(k,\xi)-{\mathcal M}_t^b(k,\xi-\eta)\big)\frac{1}{\eta}\widehat \omega(0,\eta)k\widehat\theta_{\neq}(k,\xi-\eta)
{\mathcal M}_t^b(k,\xi)\overline{\widehat \theta_{\neq}(k,\xi)}d\xi d\eta,
\end{split}
\end{equation*}
where we used $\widehat u(0,\eta)=i\eta^{-1}\widehat \omega(0,\eta)$ by \eqref{S2.3eq1}.
By Taylor's formula we have, for $k\neq0$,
$$|{\mathcal M}_t^b(k,\xi)-{\mathcal M}_t^b(k,\xi-\eta)|\leq \int_0^1|\p_\xi{\mathcal M}_t^b(k,\xi-s\eta)||\eta|ds.$$
Using the explicit expression of ${\mathcal M}_t^b$ we can show that
\begin{equation}\label{S2.3eq4}
\begin{split}
|\p_\xi{\mathcal M}_t^b(k,\xi)|&\leq C\big(\nu^\frac13|k|^{-\frac13}+\frac{1}{|k|}\big)\Lambda_t^b(k,\xi).
\end{split}
\end{equation}
Therefore, by Young's convolution inequality, we get
\begin{equation}\label{S2.3eq5}
\begin{split}
|I_{32}&|\leq\sum_{k\neq0}C(\nu^\frac13|k|^{\frac23}+1)\iint \big(\Lambda_t^b(k,\xi-\eta)+\Lambda_t^b(0,\eta)\big)
|\widehat \omega(0,\eta)|\\
&\qquad\qquad\qquad\qquad\quad\qquad\qquad\quad\times |\widehat\theta_{\neq}(k,\xi-\eta) |\Lambda_t^b(k,\xi)|\widehat \theta_{\neq}(k,\xi)|d\xi d\eta\\
&\leq C\big(\nu^\frac13\|\widehat\omega_0\|_{L^1}\||D_x|^\frac13\Lambda_t^b\theta_{\neq}\|_{L^2}^2+\nu^\frac13\|\Lambda_t^b\omega_0\|_{L^2}\|\widehat{|D_x|^\frac13\theta_{\neq}}\|_{L^1}\||D_x|^\frac13\Lambda_t^b\theta_{\neq}\|_{L^2}\\
&\qquad\qquad\qquad\qquad\qquad+\|\widehat\omega_0\|_{L^1}\|\Lambda_t^b\theta_{\neq}\|_{L^2}^2+\|\Lambda_t^b\omega_0\|_{L^2}\|\widehat\theta_{\neq}\|_{L^1}\|\Lambda_t^b\theta_{\neq}\|_{L^2}\big)\\
&\leq C\nu^\frac13\|\Lambda_t^b\omega_0\|_{L^2}\||D_x|^\frac13\Lambda_t^b\theta\|_{L^2}^2
+C\|\Lambda_t^b\omega_0\|_{L^2}\|\Lambda_t^b\theta_{\neq}\|_{L^2}^2.
\end{split}
\end{equation}
Due to ${\rm div}\ {\bf u}_{\neq}=0$, we have the cancellation
$$\langle {\bf u}_{\neq}\cdot\nabla({\mathcal M}_t^b\theta),{\mathcal M}_t^b\theta\rangle_{L^2}=0$$
and we can rewrite
$$I_{33}=\underbrace{\langle{\mathcal M}_t^b(u_{\neq}\p_x\theta)-u_{\neq}\p_x({\mathcal M}_t^b\theta),{\mathcal M}_t^b\theta\rangle_{L^2}}_{=:J}-\underbrace{\langle v_{\neq}\p_y({\mathcal M}_t^b\theta),{\mathcal M}_t^b\theta\rangle_{L^2}}_{=:J'}.$$
The term $J'$ is easy to control
\begin{equation}\label{S2.3eq6}
\begin{split}
|J'|&\leq\|v_{\neq}\|_{L^\infty}\|D_y{\mathcal M}_t^b\theta\|_{L^2}\|{\mathcal M}_t^b\theta\|_{L^2}\\
&\leq\|(-\Delta)^{-\frac12}\Lambda_t^b\omega_{\neq}\|_{L^2}\|D_y\Lambda_t^b\theta\|_{L^2}\|\Lambda_t^b\theta\|_{L^2}.
\end{split}
\end{equation}
It remains to estimate the term $J$. Noticing that $\p_x\theta_0=\p_x({\mathcal M}_t^b\theta_0)=0$, we can write
\begin{equation*}
\begin{split}
J&=\langle{\mathcal M}_t^b(u_{\neq}\p_x\theta_{\neq})-u_{\neq}\p_x({\mathcal M}_t^b\theta_{\neq}),{\mathcal M}_t^b\theta\rangle_{L^2}=J_1+J_2\\
\text{with}\quad J_1&:=\langle{\mathcal M}_t^b(u_{\neq}\p_x\theta_{\neq})-u_{\neq}\p_x({\mathcal M}_t^b\theta_{\neq}),{\mathcal M}_t^b\theta_{\neq}\rangle_{L^2},\\
J_2&:=\langle{\mathcal M}_t^b(u_{\neq}\p_x\theta_{\neq})-u_{\neq}\p_x({\mathcal M}_t^b\theta_{\neq}),{\mathcal M}_t^b\theta_0\rangle_{L^2}.
\end{split}
\end{equation*}
By Plancherel's theorem,
\beno
J_1&=&\sum_{k,l}\iint\big({\mathcal M}_t^b(k,\xi)-{\mathcal M}_t^b(k-l,\xi-\eta)\big)\\
&& \qquad\qquad \quad \widehat u_{\neq}(l,\eta)\cdot i(k-l)\widehat\theta_{\neq}(k-l,\xi-\eta)
\cdot{\mathcal M}_t^b(k,\xi)\overline{\widehat \theta_{\neq}(k,\xi)}d\xi d\eta\\
&=&-\sum_{k\neq0,l\neq0\atop k-l\neq0}\iint\big({\mathcal M}_t^b(k,\xi)-{\mathcal M}_t^b(k-l,\xi-\eta)\big)\\
&& \qquad\qquad \quad \frac{\eta(k-l)}{l^2+\eta^2}\widehat \omega_{\neq}(l,\eta) \widehat\theta_{\neq}(k-l,\xi-\eta)\cdot{\mathcal M}_t^b(k,\xi)\overline{\widehat \theta_{\neq}(k,\xi)}d\xi d\eta,
\eeno
where in the last equality we used $\widehat u_{\neq}(l,\eta)=i\eta(l^2+\eta^2)^{-1}\widehat \omega_{\neq}(l,\eta)$ by \eqref{S2.3eq1}.
In order to estimate $J_1$, the idea is to use Taylor's formula for ${\mathcal M}_t^b(k,\xi)-{\mathcal M}_t^b(k-l,\xi-\eta)$ as in the estimates of $I_{32}$. However,  ${\mathcal M}(k,\xi)$ and ${\mathcal M}_t^b(k,\xi)$ are not smooth at $k=0$. We then have to divide into four different cases:
\ben
&& A_1=\{k>0,k-l>0\},\quad
A_2=\{k<0,k-l<0\}, \notag\\
&& A_3=\{k>0,k-l<0\},\quad
A_4=\{k<0,k-l>0\}\label{S2.3eqA}
\een
and denote by
\beno
J_{1i}
&:=&-\sum_{(k,l)\in A_i}\iint\big({\mathcal M}_t^b(k,\xi)-{\mathcal M}_t^b(k-l,\xi-\eta)\big)\\
&& \qquad\qquad \qquad \frac{\eta(k-l)}{l^2+\eta^2}\widehat \omega_{\neq}(l,\eta) \widehat\theta_{\neq}(k-l,\xi-\eta) \cdot{\mathcal M}_t^b(k,\xi)\overline{\widehat \theta_{\neq}(k,\xi)}d\xi d\eta.
\eeno
We first estimate  $J_{11}$ and $J_{12}$.  When $k>0, k-l>0$, we use Taylor's formula,
\beno
|{\mathcal M}_t^b(k,\xi)-{\mathcal M}_t^b(k-l,\xi-\eta)| &\leq& \int_0^1|\p_\xi{\mathcal M}_t^b(k-sl,\xi-s\eta)||\eta|ds\\
&&+\int_0^1|\p_k{\mathcal M}_t^b(k-sl,\xi-s\eta)||l|ds.
\eeno
A direct computation gives
$$|\p_k\Lambda_t^b(k,\xi)|\leq C\Lambda_t^{b-2}(k,\xi)(|k|+|\xi+kt||t|),\qquad |t|\leq \frac{1}{|k|}\big(|\xi|+\Lambda_t(k,\xi)\big),$$
which implies
\begin{equation*}
|\p_k{\mathcal M}_t^b(k,\xi)|\leq \big(\frac{1}{k}+\frac{|\xi|}{k^2}\big)\Lambda_t^{b}(k,\xi)\quad\text{for }k>0.
\end{equation*}
Together with \eqref{S2.3eq4}, we obtain
\beno
&&|{\mathcal M}_t^b(k,\xi)-{\mathcal M}_t^b(k-l,\xi-\eta)|\\
&&\leq \int_0^1\big(\frac{\nu^\frac13|\eta|}{(k-sl)^{\frac13}}+\frac{|\eta|+|l|}{k-sl}+\frac{|\xi-s\eta||l|}{(k-sl)^2}\big)\Lambda_t^b(k-sl,\xi-s\eta)ds\\
&&\leq \big(\frac{\nu^\frac13|\eta|}{\min(k-l,k)^{\frac13}}+\frac{|\eta|+|l|}{\min(k-l,k)}+\frac{(|\xi|+|\xi-\eta|)|l|}{(k-l)k}\big)\big(\Lambda_t^b(k-l,\xi-\eta)+\Lambda_t^b(l,\eta)\big).
\eeno
Therefore, by the convolution inequality,
\beno
&& |J_{11}^{(1)}|:= \Big|\sum_{(k,l)\in A_1\atop l>0}\iint\big({\mathcal M}_t^b(k,\xi)-{\mathcal M}_t^b(k-l,\xi-\eta)\big)\\
&&\qquad\qquad\qquad \qquad  \frac{\eta(k-l)}{l^2+\eta^2}\widehat \omega_{\neq}(l,\eta) \widehat\theta_{\neq}(k-l,\xi-\eta)\cdot{\mathcal M}_t^b(k,\xi)\overline{\widehat \theta_{\neq}(k,\xi)}d\xi d\eta\Big|\\
&\leq& \sum_{(k,l)\in A_1\atop l>0}\iint\Big(\frac{\nu^\frac13|\eta|}{(k-l)^{\frac13}}+\frac{|\eta|+|l|}{k-l}+\frac{(|\xi|+|\xi-\eta|)|l|}{(k-l)k}\Big)\big(\Lambda_t^b(k-l,\xi-\eta)+\Lambda_t^b(l,\eta)\big)\\
&&\qquad\qquad\quad\qquad\cdot\frac{|\eta|(k-l)}{l^2+\eta^2}|\widehat \omega_{\neq}(l,\eta)\widehat\theta_{\neq}(k-l,\xi-\eta){\Lambda}_t^b(k,\xi)\widehat \theta_{\neq}(k,\xi)|d\xi d\eta\\
&\leq& \sum_{(k,l)\in A_1\atop l>0}\iint\big(\nu^\frac13(k-l)^{\frac23}+1+\frac{|\xi|+|\xi-\eta|}{(l^2+\eta^2)^{\frac12}}\big)\big(\Lambda_t^b(k-l,\xi-\eta)+\Lambda_t^b(l,\eta)\big)\\
&&\qquad\qquad\qquad\qquad\cdot|\widehat \omega_{\neq}(l,\eta)\widehat\theta_{\neq}(k-l,\xi-\eta){\Lambda}_t^b(k,\xi)\widehat \theta_{\neq}(k,\xi)|d\xi d\eta\\
&\leq& \|\widehat\omega_{\neq}\|_{L^1}\big(\nu^\frac13\||D_x|^\frac13\Lambda_t^b\theta_{\neq}\|_{L^2}^2+\|\Lambda_t^b\theta_{\neq}\|_{L^2}^2\big)\\
&& \qquad +\|\widehat{(-\Delta)^{-\frac12}\omega_{\neq}}\|_{L^1}\|\Lambda_t^b\theta_{\neq}\|_{L^2}\|D_y\Lambda_t^b\theta_{\neq}\|_{L^2}\\
&&\qquad+\|\Lambda_t^b\omega_{\neq}\|_{L^2}\big(\nu^\frac13\|\widehat{{|D_x|^\frac13}\theta_{\neq}}\|_{L^1}\||D_x|^\frac13\Lambda_t^b\theta_{\neq}\|_{L^2}+\|\widehat{\theta}_{\neq}\|_{L^1}\|\Lambda_t^b\theta_{\neq}\|_{L^2}\big)\\
&&\qquad+\|(-\Delta)^{-\frac12}\Lambda_t^b\omega_{\neq}\|_{L^1}\big(\|\widehat\theta_{\neq}\|_{L^1}\|D_y\Lambda_t^b\theta_{\neq}\|_{L^2}+\|\widehat{D_y\theta_{\neq}}\|_{L^1}\|\Lambda_t^b\theta_{\neq}\|_{L^2}\big)\\
&\leq&  \|\Lambda_t^b\omega_{\neq}\|_{L^2}\||D_x|^\frac13\Lambda_t^b\theta_{\neq}\|_{L^2}^2
+\|(-\Delta)^{-\frac12}\Lambda_t^b\omega_{\neq}\|_{L^2}\|\Lambda_t^b\theta_{\neq}\|_{L^2}\|D_y\Lambda_t^b\theta_{\neq}\|_{L^2},
\eeno
where we have used that $(k-l)^\frac23\leq(k-l)^\frac13k^\frac13$ for $k>0,k-l>0,l>0$.
On the other hand, when $k\geq1,l<0$, we have the inequalities
\beno
&&\frac{k-l}{k^\frac13}\leq \min\big((k-l)^\frac13k^\frac13+(k-l)^\frac13|l|^{\frac23},2(k-l)^{\frac23}|l|^\frac13\big), \\
&& \frac{k-l}{k}\leq 2\min\big((k-l)^{\frac13}|l|^{\frac23},(k-l)^\frac23|l|^\frac13\big).
\eeno
$J_{11}^{(2)}$ can be estimated as follows,
\beno
&&|J_{11}^{(2)}|
:=\Big|\sum_{(k,l)\in A_1\atop l<0}\iint\big({\mathcal M}_t^b(k,\xi)-{\mathcal M}_t^b(k-l,\xi-\eta)\big)\\
&&\qquad \qquad \qquad \frac{\eta(k-l)}{l^2+\eta^2}\widehat \omega_{\neq}(l,\eta) \widehat\theta_{\neq}(k-l,\xi-\eta)\cdot{\mathcal M}_t^b(k,\xi)\overline{\widehat \theta_{\neq}(k,\xi)}d\xi d\eta\Big|\\
&\leq& \sum_{(k,l)\in A_1\atop l<0}\iint\Big(\frac{\nu^\frac13|\eta|}{k^{\frac13}}+\frac{|\eta|+|l|}{k}+\frac{(|\xi|+|\xi-\eta|)|l|}{(k-l)k}\Big)\big(\Lambda_t^b(k-l,\xi-\eta)+\Lambda_t^b(l,\eta)\big)\\
&&\qquad\qquad\qquad\cdot\frac{|\eta|(k-l)}{l^2+\eta^2}|\widehat \omega_{\neq}(l,\eta)\widehat\theta_{\neq}(k-l,\xi-\eta){\Lambda}_t^b(k,\xi)\widehat \theta_{\neq}(k,\xi)|d\xi d\eta\\
&\leq& \sum_{(k,l)\in A_1\atop l<0}\iint\Big(\big(\nu^\frac13(k-l)^\frac13k^{\frac13}+|\xi|+|\xi-\eta|\big)\big(\Lambda_t^b(k-l,\xi-\eta)+\Lambda_t^b(l,\eta)\big)\\
&&\qquad\qquad\quad+(k-l)^\frac13|l|^\frac23\Lambda_t^{b}(k-l,\xi-\eta)+(k-l)^\frac23|l|^\frac13\Lambda_t^{b}(l,\eta)\Big)\\
&&\qquad\qquad\qquad\qquad\qquad\cdot|\widehat \omega_{\neq}(l,\eta)\widehat\theta_{\neq}(k-l,\xi-\eta){\Lambda}_t^b(k,\xi)\widehat \theta_{\neq}(k,\xi)|d\xi d\eta\\
&\leq& \nu^\frac13\|\widehat\omega_{\neq}\|_{L^1}\||D_x|^\frac13\Lambda_t^b\theta_{\neq}\|_{L^2}^2
+\|\widehat{|D_x|^\frac23\omega_{\neq}}\|_{L^1}\||D_x|^\frac13\Lambda_t^b\theta_{\neq}\|_{L^2}\|\Lambda_t^b\theta_{\neq}\|_{L^2}\\
&&\qquad+\|\widehat\omega_{\neq}\|_{L^1}\|\Lambda_t^b\theta_{\neq}\|_{L^2}\|D_y\Lambda_t^b\theta_{\neq}\|_{L^2}
+\nu^\frac13\|\Lambda_t^b\omega_{\neq}\|_{L^2}\|\widehat{|D_x|^\frac13\theta_{\neq}}\|_{L^2}\||D_x|^\frac13\Lambda_t^b\theta_{\neq}\|_{L^2}\\
&&\qquad+\||D_x|^\frac13\Lambda_t^b\omega_{\neq}\|_{L^2}\|\widehat{|D_x|^\frac23\theta_{\neq}}\|_{L^1}\|\Lambda_t^b\theta_{\neq}\|_{L^2}\\
&&\qquad+\|\Lambda_t^b\omega_{\neq}\|_{L^2}\big(\|\widehat\theta_{\neq}\|_{L^1}\|D_y\Lambda_t^b\theta_{\neq}\|_{L^2}+\|\widehat{D_y\theta_{\neq}}\|_{L^1}\|\Lambda_t^b\theta_{\neq}\|_{L^2}\big)\\
&\leq& \|\Lambda_t^b\omega_{\neq}\|_{L^2}\big(\nu^\frac13\||D_x|^\frac13\Lambda_t^b\theta_{\neq}\|_{L^2}^2+\|\Lambda_t^b\theta_{\neq}\|_{L^2}\|D_y\Lambda_t^b\theta_{\neq}\|_{L^2}\big)\\
&&\qquad+\||D_x|^\frac13\Lambda_t^b\omega_{\neq}\|_{L^2}\||D_x|^\frac13\Lambda_t^b\theta_{\neq}\|_{L^2}\|\Lambda_t^b\theta_{\neq}\|_{L^2},
\eeno
where we have used
$$
\|\widehat{|D_x|^\frac23\omega_{\neq}}\|_{L^1}\leq \||D_x|^\frac13\Lambda_t^b\omega_{\neq}\|_{L^2},\qquad\text{provided that }b>\frac43.
$$
Combining the bounds for $J_{11}^{(1)}$ and $J_{11}^{(2)}$ yields
\begin{equation*}
\begin{split}
|J_{11}|&\leq  \|\Lambda_t^b\omega_{\neq}\|_{L^2}\big(\||D_x|^\frac13\Lambda_t^b\theta_{\neq}\|_{L^2}^2+\|\Lambda_t^b\theta_{\neq}\|_{L^2}\|D_y\Lambda_t^b\theta_{\neq}\|_{L^2}\big)\\
&\quad +\||D_x|^\frac13\Lambda_t^b\omega_{\neq}\|_{L^2}\||D_x|^\frac13\Lambda_t^b\theta_{\neq}\|_{L^2}\|\Lambda_t^b\theta_{\neq}\|_{L^2}.
\end{split}
\end{equation*}
The term $J_{12}$ can be treated in the same way. To estimate  $J_{13}$ and $J_{14}$, we notice that, when $k>0, k-l<0$ or $k<0,k-l>0$, we have $|k-l|<|l|$ and thus
\begin{equation*}
\begin{split}
|J_{13}+J_{14}|
&\leq\sum_{(k,l)\in A_3\cup A_4}\iint\big(\Lambda_t^b(k-l,\xi-\eta)+\Lambda_t^b(l,\eta)\big)\frac{|\eta||k-l|}{l^2+\eta^2}|\widehat \omega_{\neq}(l,\eta)| \\
&\qquad\qquad\qquad\qquad\qquad\quad\cdot|\widehat\theta_{\neq}(k-l,\xi-\eta)\Lambda_t^b(k,\xi)\widehat \theta_{\neq}(k,\xi)|d\xi d\eta\\
&\leq \|\widehat\omega_{\neq}\|_{L^1}\|\Lambda_t^b\theta_{\neq}\|_{L^2}^2+\|\widehat\theta_{\neq}\|_{L^1}\|\Lambda_t^b\omega_{\neq}\|_{L^2}\|\Lambda_t^b\theta_{\neq}\|_{L^2}\\
&\leq \|\Lambda_t^b\omega_{\neq}\|_{L^2}\|\Lambda_t^b\theta_{\neq}\|_{L^2}^2.
\end{split}
\end{equation*}
This finishes the estimate for $J_1$,
\begin{equation*}\label{S2.3eq7}
\begin{split}
|J_1|&\leq  \|\Lambda_t^b\omega_{\neq}\|_{L^2}\big(\||D_x|^\frac13\Lambda_t^b\theta_{\neq}\|_{L^2}^2+\|\Lambda_t^b\theta_{\neq}\|_{L^2}\|D_y\Lambda_t^b\theta_{\neq}\|_{L^2}\big)\\
&\qquad+\||D_x|^\frac13\Lambda_t^b\omega_{\neq}\|_{L^2}\||D_x|^\frac13\Lambda_t^b\theta_{\neq}\|_{L^2}\|\Lambda_t^b\theta_{\neq}\|_{L^2}.
\end{split}
\end{equation*}
To estimate  $J_2$, we apply Plancherel's theorem and \eqref{S2.3eq1} to write
\begin{equation*}
\begin{split}
|J_2|&=\Big|\sum_{l\neq0}\iint\big({\mathcal M}_t^b(0,\xi)-{\mathcal M}_t^b(-l,\xi-\eta)\big)\widehat u_{\neq}(l,\eta) i(-l)\widehat\theta_{\neq}(-l,\xi-\eta)\overline{{\mathcal M}_t^b\widehat \theta(0,\xi)}d\xi d\eta\Big|\\
&\leq\sum_{l\neq0}\iint\big(\Lambda_t^b(-l,\xi-\eta)+\Lambda_t^b(l,\eta)\big)\frac{|l||\eta|}{l^2+\eta^2}|\widehat\omega_{\neq}(l,\eta)\widehat\theta_{\neq}(-l,\xi-\eta)\Lambda_t^b\widehat \theta(0,\xi)|d\xi d\eta\\
&\leq \|\widehat\omega_{\neq}\|_{L^1}\|\Lambda_t^b\theta_{\neq}\|_{L^2}\|\Lambda_t^b\theta_0\|_{L^2}+\|\widehat\theta_{\neq}\|_{L^1}\|\Lambda_t^b\omega_{\neq}\|_{L^2}\|\Lambda_t^b\theta_{0}\|_{L^2}\\
&\leq \|\Lambda_t^b\omega_{\neq}\|_{L^2}\|\Lambda_t^b\theta_{\neq}\|_{L^2}\|\Lambda_t^b\theta_0\|_{L^2}.
\end{split}
\end{equation*}
Combining the bounds for $J_1$ and $J_2$, we obtain
\begin{equation*}
\begin{split}
|J|&\leq  \|\Lambda_t^b\omega_{\neq}\|_{L^2}\big(\||D_x|^\frac13\Lambda_t^b\theta_{\neq}\|_{L^2}^2+\|\Lambda_t^b\theta_{\neq}\|_{L^2}\|D_y\Lambda_t^b\theta_{\neq}\|_{L^2}+\|\Lambda_t^b\theta_{\neq}\|_{L^2}\|\Lambda_t^b\theta_0\|_{L^2}\big)\\
&\qquad+\||D_x|^\frac13\Lambda_t^b\omega_{\neq}\|_{L^2}\||D_x|^\frac13\Lambda_t^b\theta_{\neq}\|_{L^2}\|\Lambda_t^b\theta_{\neq}\|_{L^2}.
\end{split}
\end{equation*}
Together with \eqref{S2.3eq6}, we finish the estimates for $I_{33}$:
\begin{equation}\label{S2.3eq8}
\begin{split}
|I_{33}| &\leq \|\Lambda_t^b\omega_{\neq}\|_{L^2}\||D_x|^\frac13\Lambda_t^b\theta\|_{L^2}^2+\||D_x|^\frac13\Lambda_t^b\omega\|_{L^2}\||D_x|^\frac13\Lambda_t^b\theta\|_{L^2}\|\Lambda_t^b\theta_{\neq}\|_{L^2}\\
&\qquad
+\|\Lambda_t^b\omega_{\neq}\|_{L^2}\|\Lambda_t^b\theta\|_{L^2}\big(\|D_y\Lambda_t^b\theta\|_{L^2}+\|\Lambda_t^b\theta_{\neq}\|_{L^2}\big).
\end{split}
\end{equation}
It follows from \eqref{S2.3eq3}, \eqref{S2.3eq5} and \eqref{S2.3eq8} that
\begin{equation}\label{S2.3eq9}
\begin{split}
|I_{3}|
&\leq \|\Lambda_t^b\omega\|_{L^2}\||D_x|^\frac13\Lambda_t^b\theta\|_{L^2}^2+\||D_x|^\frac13\Lambda_t^b\omega\|_{L^2}\||D_x|^\frac13\Lambda_t^b\theta\|_{L^2}\|\Lambda_t^b\theta_{\neq}\|_{L^2}\\
&\qquad
+\|\Lambda_t^b\omega_{\neq}\|_{L^2}\|\Lambda_t^b\theta\|_{L^2}\big(\|D_y\Lambda_t^b\theta\|_{L^2}+\|\Lambda_t^b\theta_{\neq}\|_{L^2}\big).
\end{split}
\end{equation}
Similarly, the upper bound for $I_2$ is given by
\begin{equation}\label{S2.3eq10}
\begin{split}
|I_{2}|
&\leq   \|\Lambda_t^b\omega\|_{L^2}\||D_x|^\frac13\Lambda_t^b\omega\|_{L^2}^2+\|\Lambda_t^b\omega\|_{L^2}\|\Lambda_t^b\omega_{\neq}\|_{L^2}\|D_y\Lambda_t^b\omega\|_{L^2}.
\end{split}
\end{equation}
{\bf Estimates for $I_4$}. \quad As in the estimates of $I_3$, we decompose the term $I_4$ as
\begin{equation*}
\begin{split}
I_4&=I_{41}+I_{42}+I_{43} \quad\text{with}\\
I_{41}&:=\langle\Lambda_t^b\big(v_{\neq}\p_y\theta\big),|D_x|^\frac23{\mathcal M}\Lambda_t^b\theta\rangle_{L^2},\\
I_{42}&:=\langle\Lambda_t^b\big(u_0\p_x\theta\big),|D_x|^\frac23{\mathcal M}\Lambda_t^b\theta\rangle_{L^2},\\
I_{43}&:=\langle\Lambda_t^b\big(u_{\neq}\p_x\theta\big),|D_x|^\frac23{\mathcal M}\Lambda_t^b\theta\rangle_{L^2}.
\end{split}
\end{equation*}
By Lemma \ref{Lam},
\begin{equation}\label{S2.4eq1}
\begin{split}
&I_{41}\leq \||D_x|^\frac13\Lambda_t^b(v_{\neq}\p_y\theta)\|_{L^2}\||D_x|^\frac13\Lambda_t^b\theta\|_{L^2}\\
&\leq \big(\||D_x|^\frac13\Lambda_t^bv_{\neq}\|_{L^2}\|D_y\Lambda_t^b\theta\|_{L^2}+\|\Lambda_t^bv_{\neq}\|_{L^2}\||D_x|^\frac13D_y\Lambda_t^b\theta\|_{L^2}\big)\||D_x|^\frac13\Lambda_t^b\theta\|_{L^2}.
\end{split}
\end{equation}
Setting ${\mathcal N}_t^b(k,\xi):=|k|^\frac13{\mathcal M}_t^b(k,\xi)$, we can write
\begin{equation*}
\begin{split}
I_{42}&=\langle{\mathcal N}_t^b(u_0\p_x\theta_{\neq})-u_0\p_x{\mathcal N}_t^b\theta_{\neq},{\mathcal N}_t^b\theta_{\neq}\rangle_{L^2}.
\end{split}
\end{equation*}
The estimates for $I_{42}$ are similar to those for $I_{32}$,
\begin{equation}\label{S2.4eq2}
|I_{42}|\leq \nu^\frac13\|\Lambda_t^b\omega_0\|_{L^2}\||D_x|^\frac23\Lambda_t^b\theta_{\neq}\|_{L^2}^2+\|\Lambda_t^b\omega_0\|_{L^2}\||D_x|^\frac13\Lambda_t^b\theta_{\neq}\|_{L^2}^2.
\end{equation}
In order to estimate the term $I_{43}$, we decompose it as
\begin{equation*}
\begin{split}
I_{43}&=\langle{\mathcal N}_t^b(u_{\neq} \p_x\theta),{\mathcal N_t^b\theta}\rangle_{L^2}=K+K'\\
\text{with}\qquad K&=\langle{\mathcal N}_t^b(u_{\neq} \p_x\theta)-u_{\neq}\p_x({\mathcal N}_t^b\theta),{\mathcal N_t^b\theta}\rangle_{L^2},\\
K'&=-\langle v_{\neq}\p_y({\mathcal N}_t^b\theta),{\mathcal N}_t^b\theta\rangle_{L^2}.
\end{split}
\end{equation*}
The term $K'$ can be bounded easily,
\begin{equation}\label{S2.4eq3}
\begin{split}
|K'|&\leq \|v_{\neq}\|_{L^\infty}\|\p_y{\mathcal N}_t^b\theta\|_{L^2}\|{\mathcal N}_t^b\theta\|_{L^2}\\
&\leq \|(-\Delta)^{-\frac12}\Lambda_t^b\omega_{\neq}\|_{L^2}\|D_y|D_x|^\frac13\Lambda_t^b\theta\|_{L^2}\||D_x|^\frac13\Lambda_t^b\theta\|_{L^2}.
\end{split}
\end{equation}
For the term $K$, due to $\p_x\theta_0= \p_x{\mathcal N}_t^b\theta_0=0$,
\begin{equation*}
\begin{split}
K&=\langle{\mathcal N}_t^b(u_{\neq} \p_x\theta_{\neq})-u_{\neq}\p_x({\mathcal N}_t^b\theta_{\neq}),{\mathcal N_t^b\theta_{\neq}}\rangle_{L^2}.
\end{split}
\end{equation*}
By Plancherel's theorem and \eqref{S2.3eq1},
\begin{equation*}
\begin{split}
K&=-\sum_{k,l}\iint\big({\mathcal N}_t^b(k,\xi)-{\mathcal N}_t^b(k-l,\xi-\eta)\big)\frac{\eta(k-l)}{l^2+\eta^2}\widehat\omega_{\neq}(l,\eta)\\
&\qquad\qquad\qquad\qquad\qquad\cdot\widehat\theta_{\neq}(k-l,\xi-\eta){\mathcal N}_t^b(k,\xi)\overline{\widehat\theta_{\neq}(k,\xi)}d\xi d\eta\\
&=K_1+K_2+K_3+K_4,
\end{split}
\end{equation*}
where, for $i=1,2,3,4$,
\begin{equation*}
\begin{split}
K_i&=-\sum_{(k,l)\in A_i}\iint\big({\mathcal N}_t^b(k,\xi)-{\mathcal N}_t^b(k-l,\xi-\eta)\big)\frac{\eta(k-l)}{l^2+\eta^2}\widehat\omega_{\neq}(l,\eta)\\
&\qquad\qquad\qquad\qquad\qquad\cdot\widehat\theta_{\neq}(k-l,\xi-\eta){\mathcal N}_t^b(k,\xi)\overline{\widehat\theta_{\neq}(k,\xi)}d\xi d\eta
\end{split}
\end{equation*}
with $A_i$ defined in \eqref{S2.3eqA}. For any $k\neq0$,
\beno
&&|\p_\xi{\mathcal N}_t^b(k,\xi)|\leq (\nu^\frac13+|k|^{-\frac23})\Lambda_t^b(k,\xi),\\
&& |\p_{k}{\mathcal N}_t^b(k,\xi)|\leq (|k|^{-\frac23}+|k|^{-\frac53}|\xi|)\Lambda_t^{b}(k,\xi).
\eeno
When $k>0,k-l>0$, using Taylor's formula, we have
\begin{equation*}
\begin{split}
&|{\mathcal N}_t^b(k,\xi)-{\mathcal N}_t^b(k-l,\xi-\eta)|\\
&\qquad\leq \big(\nu^\frac13|\eta|+\frac{|l|+|\xi-\eta|+|\xi|}{\min(k-l,k)^\frac23}\big)\big(\Lambda_t^b(k-l,\xi-\eta)+\Lambda_t^b(l,\eta)\big).
\end{split}
\end{equation*}
Therefore, by the convolution inequality,
\beno
|K_1^{(1)}|&=& \Big|\sum_{(k,l)\in A_1\atop l>0}\iint\big({\mathcal N}_t^b(k,\xi)-{\mathcal N}_t^b(k-l,\xi-\eta)\big)\frac{\eta(k-l)}{l^2+\eta^2}\widehat\omega_{\neq}(l,\eta)\\
&&\qquad\qquad\qquad\qquad\qquad\quad\cdot\widehat\theta_{\neq}(k-l,\xi-\eta){\mathcal N}_t^b(k,\xi)\overline{\widehat\theta_{\neq}(k,\xi)}d\xi d\eta\Big|\\
&\leq&\sum_{(k,l)\in A_1\atop l>0}\iint\big(\nu^\frac13|\eta|+\frac{|l|+|\xi-\eta|+|\xi|}{(k-l)^\frac23}\big)\big(\Lambda_t^b(k-l,\xi-\eta)+\Lambda_t^b(l,\eta)\big)\\
&&\qquad\qquad\qquad\qquad\cdot\frac{\eta(k-l)}{l^2+\eta^2}|\widehat\omega_{\neq}(l,\eta)\widehat\theta_{\neq}(k-l,\xi-\eta){\mathcal N}_t^b\widehat\theta_{\neq}(k,\xi)|d\xi d\eta\\
&\leq&\sum_{(k,l)\in A_1\atop l>0}\iint\big(\nu^\frac13(k-l)+(k-l)^\frac13+(|\xi-\eta|+|\xi|)(k-l)^\frac13\big)\\
&&\qquad\cdot\big(\Lambda_t^b(k-l,\xi-\eta)+\Lambda_t^b(l,\eta)\big)
|\widehat\omega_{\neq}(l,\eta)\widehat\theta_{\neq}(k-l,\xi-\eta){\mathcal N}_t^b\widehat\theta_{\neq}(k,\xi)|d\xi d\eta\\
&\leq& \|\widehat \omega_{\neq}\|_{L^1}\big(\nu^\frac13\||D_x|^\frac23\Lambda_t^b\theta_{\neq}\|_{L^2}\||D_x|^\frac13{\mathcal N}_t^b\theta_{\neq}\|_{L^2}+\||D_x|^\frac13\Lambda_t^b\theta_{\neq}\|_{L^2}\|{\mathcal N}_t^b\theta_{\neq}\|_{L^2}\\
&&\qquad\qquad\qquad+\|D_y|D_x|^\frac13\Lambda_t^b\theta_{\neq}\|_{L^2}\|{\mathcal N}_t^b\theta_{\neq}\|_{L^2}+\||D_x|^\frac13\Lambda_t^b\theta_{\neq}\|_{L^2}\|D_y{\mathcal N}_t^b\theta_{\neq}\|_{L^2}\big)\\
&&\qquad+\|\Lambda_t^b\omega_{\neq}\|_{L^2}\big(\nu^\frac13\|\widehat{|D_x|^\frac23\theta_{\neq}}\|_{L^1}\||D_x|^\frac13{\mathcal N}_t^b\theta_{\neq}\|_{L^2}+\|\widehat{|D_x|^\frac13\theta_{\neq}}\|_{L^1}\|{\mathcal N}_t^b\theta_{\neq}\|_{L^2}\\
&&\qquad\qquad\qquad\qquad+\|\widehat{D_y|D_x|^\frac13\theta_{\neq}}\|_{L^1}\|{\mathcal N}_t^b\theta_{\neq}\|_{L^2}+\|\widehat{|D_x|^\frac13\theta_{\neq}}\|_{L^1}\|D_y{\mathcal N}_t^b\theta_{\neq}\|_{L^2}\big)\\
&\leq& \|\Lambda_t^b\omega_{\neq}\|_{L^2}\big(\||D_x|^\frac23\Lambda_t^b\theta_{\neq}\|_{L^2}^2+\|D_y|D_x|^\frac13\Lambda_t^b\theta_{\neq}\|_{L^2}\||D_x|^\frac13\Lambda_t^b\theta_{\neq}\|_{L^2}\big),
\eeno
where we used $k-l\leq (k-l)^\frac23k^\frac13$ for $k>0,k-l>0,l>0$. Using the fact that, for $k>0, l<0$,
$$k-l\leq (k-l)^\frac23(k^\frac13+|l|^\frac13),\quad \frac{k-l}{k^\frac23}\leq 2(k-l)^\frac23|l|^\frac13,\quad
\frac{k-l}{k^\frac23|l|}\leq 2(k-l)^\frac13, $$
we have
\beno
|K_1^{(2)}|&=&\Big|\sum_{(k,l)\in A_1\atop l<0}\iint\big({\mathcal N}_t^b(k,\xi)-{\mathcal N}_t^b(k-l,\xi-\eta)\big)\frac{\eta(k-l)}{l^2+\eta^2}\widehat\omega_{\neq}(l,\eta)\\
&&\qquad\qquad\qquad\qquad\quad\qquad\cdot\widehat\theta_{\neq}(k-l,\xi-\eta){\mathcal N}_t^b(k,\xi)\overline{\widehat\theta_{\neq}(k,\xi)}d\xi d\eta\Big|\\
&\leq&\sum_{(k,l)\in A_1\atop l>0}\iint\big(\nu^\frac13|\eta|+\frac{|l|+|\xi-\eta|+|\xi|}{k^\frac23}\big)\big(\Lambda_t^b(k-l,\xi-\eta)+\Lambda_t^b(l,\eta)\big)\\
&&\qquad\qquad\qquad\qquad\qquad\cdot\frac{\eta(k-l)}{l^2+\eta^2}|\widehat\omega_{\neq}(l,\eta)\widehat\theta_{\neq}(k-l,\xi-\eta){\mathcal N}_t^b\widehat\theta_{\neq}(k,\xi)|d\xi d\eta\\
&\leq& \sum_{(k,l)\in A_1\atop l>0}\iint\big(\nu^\frac13(k-l)^\frac23k^\frac13+(k-l)^\frac23|l|^\frac13+(|\xi-\eta|+|\xi|)(k-l)^\frac13\big)\\
&&\qquad\quad\cdot\big(\Lambda_t^b(k-l,\xi-\eta)+\Lambda_t^b(l,\eta)\big)
|\widehat\omega_{\neq}(l,\eta)\widehat\theta_{\neq}(k-l,\xi-\eta){\mathcal N}_t^b\widehat\theta_{\neq}(k,\xi)|d\xi d\eta\\
&\leq& \nu^\frac13\big(\||D_x|^\frac23\Lambda_t^b\theta_{\neq}\|_{L^2}\|\widehat\omega_{\neq}\|_{L^1}\||D_x|^\frac13{\mathcal N}_t^b\theta_{\neq}\|_{L^2}
+\|\Lambda_t^b\omega_{\neq}\|_{L^2}\|\widehat{|D_x|^\frac23\theta_{\neq}}\|_{L^1}\||D_x|^\frac13{\mathcal N}_t^b\theta_{\neq}\|_{L^2}\big)\\
&&\qquad+\||D_x|^\frac23\Lambda_t^b\theta_{\neq}\|_{L^2}\|\widehat{|D_x|^\frac13\omega_{\neq}}\|_{L^1}\|{\mathcal N}_t^b\theta_{\neq}\|_{L^2}
+\||D_x|^\frac13\Lambda_t^b\omega_{\neq}\|_{L^2}\|\widehat{|D_x|^\frac23\theta_{\neq}}\|_{L^1}\|{\mathcal N}_t^b\theta_{\neq}\|_{L^2}\\
&&\qquad+\|\widehat\omega_{\neq}\|_{L^1}\big(\|D_y|D_x|^\frac13\Lambda_t^b\theta_{\neq}\|_{L^2}\|{\mathcal N}_t^b\theta_{\neq}\|_{L^2}+\||D_x|^\frac13\Lambda_t^b\theta_{\neq}\|_{L^2}\|D_y{\mathcal N}_t^b\theta_{\neq}\|_{L^2}\big)\\
&&\qquad+\|\Lambda_t^b\omega_{\neq}\|_{L^2}\big(\|\widehat{D_y|D_x|^\frac13\theta_{\neq}}\|_{L^1}\|{\mathcal N}_t^b\theta_{\neq}\|_{L^2}+\|\widehat{|D_x|^\frac13\theta_{\neq}}\|_{L^1}\|D_y{\mathcal N}_t^b\theta_{\neq}\|_{L^2}\big)\\
&\leq& \|\Lambda_t^b\omega_{\neq}\|_{L^2}\big(\nu^\frac13\||D_x|^\frac23\Lambda_t^b\theta_{\neq}\|_{L^2}^2+\|D_y|D_x|^\frac13\Lambda_t^b\theta_{\neq}\|_{L^2}\||D_x|^\frac13\Lambda_t^b\theta_{\neq}\|_{L^2}\big)\\
&&\qquad+\||D_x|^\frac13\Lambda_t^b\omega_{\neq}\|_{L^2}\||D_x|^\frac23\Lambda_t^b\theta_{\neq}\|_{L^2}\||D_x|^\frac13\Lambda_t^b\theta_{\neq}\|_{L^2}.
\eeno
This completes the estimates for $K_1$,
\begin{equation*}
\begin{split}
K_1&\leq\|\Lambda_t^b\omega_{\neq}\|_{L^2}\big(\||D_x|^\frac23\Lambda_t^b\theta_{\neq}\|_{L^2}^2+\|D_y|D_x|^\frac13\Lambda_t^b\theta_{\neq}\|_{L^2}\||D_x|^\frac13\Lambda_t^b\theta_{\neq}\|_{L^2}\big)\\
&\qquad+\||D_x|^\frac13\Lambda_t^b\omega_{\neq}\|_{L^2}\||D_x|^\frac23\Lambda_t^b\theta_{\neq}\|_{L^2}\||D_x|^\frac13\Lambda_t^b\theta_{\neq}\|_{L^2}.
\end{split}
\end{equation*}
We can estimate the term $K_2$ in the same way. To estimate $K_3$ and $K_4$, we
notice that when $k>0,k-l<0$ or $k<0, k-l>0$, $|k-l|<|l|$. Therefore,
\beno
|K_3+K_4| &\leq& \sum_{(k,l)\in A_3\cup A_4}\iint\big({\mathcal N}_t^b(k,\xi)+{\mathcal N}_t^b(k-l,\xi-\eta)\big)\\
&& \qquad \qquad\qquad |\widehat \omega_{\neq}(l,\eta)\widehat{\theta_{\neq}}(k-l,\xi-\eta){\mathcal N}_t^b\widehat{\theta_{\neq}}(k,\xi)|d\xi d\eta\\
&\leq& \sum_{(k,l)\in A_3\cup A_4}\iint\big((|k|^\frac13+|k-l|^\frac13)\Lambda_t^b(k-l,\xi-\eta)+|k|^\frac13\Lambda_t^b(k,\xi)\big)\\
&&\qquad\qquad\qquad |\widehat \omega_{\neq}(l,\eta)\widehat{\theta_{\neq}}(k-l,\xi-\eta){\mathcal N}_t^b\widehat\theta_{\neq}(k,\xi)|d\xi d\eta\\
&\leq& \|\widehat\omega_{\neq}\|_{L^1}\big(\|\Lambda_t^b\theta_{\neq}\|_{L^2}\||D_x|^\frac13{\mathcal N}_t^b\theta_{\neq}\|_{L^2}+\||D_x|^\frac13\Lambda_t^b\theta_{\neq}\|_{L^2}\|{\mathcal N}_t^b\theta_{\neq}\|_{L^2}\big)\\
&&\qquad\qquad+\|\Lambda_t^b\omega_{\neq}\|_{L^2}\|\widehat\theta_{\neq}\|_{L^1}\||D_x|^\frac13{\mathcal N}_t^b\theta_{\neq}\|_{L^2}\\
&\leq& \|\Lambda_t^b\omega_{\neq}\|_{L^2}\big(\|\Lambda_t^b\theta_{\neq}\|_{L^2}\||D_x|^\frac23\Lambda_t^b\theta_{\neq}\|_{L^2}+\||D_x|^\frac13\Lambda_t^b\theta_{\neq}\|_{L^2}^2\big).
\eeno
Summarizing the estimates, we achieve that
\begin{equation*}\label{S2.4eq4}
\begin{split}
|K|&\leq\|\Lambda_t^b\omega_{\neq}\|_{L^2}\big(\||D_x|^\frac23\Lambda_t^b\theta_{\neq}\|_{L^2}^2+\|D_y|D_x|^\frac13\Lambda_t^b\theta_{\neq}\|_{L^2}\||D_x|^\frac13\Lambda_t^b\theta_{\neq}\|_{L^2}\big)\\
&\qquad+\||D_x|^\frac13\Lambda_t^b\omega_{\neq}\|_{L^2}\||D_x|^\frac23\Lambda_t^b\theta_{\neq}\|_{L^2}\||D_x|^\frac13\Lambda_t^b\theta_{\neq}\|_{L^2}.
\end{split}
\end{equation*}
Together with \eqref{S2.4eq3}, we obtain
\begin{equation}\label{S2.4eq5}
\begin{split}
|I_{43}|&\leq\|\Lambda_t^b\omega_{\neq}\|_{L^2}\big(\||D_x|^\frac23\Lambda_t^b\theta\|_{L^2}^2+\|D_y|D_x|^\frac13\Lambda_t^b\theta\|_{L^2}\||D_x|^\frac13\Lambda_t^b\theta\|_{L^2}\big)\\
&\qquad+\||D_x|^\frac13\Lambda_t^b\omega\|_{L^2}\||D_x|^\frac23\Lambda_t^b\theta\|_{L^2}\||D_x|^\frac13\Lambda_t^b\theta\|_{L^2}.
\end{split}
\end{equation}
Then by \eqref{S2.4eq1}, \eqref{S2.4eq2} and \eqref{S2.4eq5}, we finish the estimates for $I_{4}$,
\begin{equation}\label{S2.4eq6}
\begin{split}
|I_{4}|&\leq\|\Lambda_t^b\omega\|_{L^2}\||D_x|^\frac23\Lambda_t^b\theta\|_{L^2}^2+\|\Lambda_t^b\omega_{\neq}\|_{L^2}\|D_y|D_x|^\frac13\Lambda_t^b\theta\|_{L^2}\||D_x|^\frac13\Lambda_t^b\theta\|_{L^2}\\
&\qquad+ \||D_x|^\frac13(-\Delta)^{-\frac12}\Lambda_t^b\omega_{\neq}\|_{L^2}\|D_y\Lambda_t^b\theta\|_{L^2}\||D_x|^\frac13\Lambda_t^b\theta\|_{L^2}\\
&\qquad+\||D_x|^\frac13\Lambda_t^b\omega\|_{L^2}\||D_x|^\frac23\Lambda_t^b\theta\|_{L^2}\||D_x|^\frac13\Lambda_t^b\theta\|_{L^2}.
\end{split}
\end{equation}
Integrating \eqref{S2.2eq5}, \eqref{S2.2eq6} and \eqref{S2.2eq7} in time and making use of
the upper bounds in  \eqref{S2.2eq8}, \eqref{S2.3eq9}, \eqref{S2.3eq10} and \eqref{S2.4eq6}, we obtain,  for $b>\frac43$,
\ben
&&\|\Lambda_t^b\omega\|_{L_t^\infty(L^2)}^2+\nu\|D_y\Lambda_t^b\omega\|_{L_t^2(L^2)}^2
+\frac18\nu^\frac13\| |D_x|^\frac13\Lambda_t^b\omega\|_{L_t^2(L^2)}^2+\|(-\Delta)^{-\frac12}\Lambda_t^b\omega_{\neq}\|_{L_t^2(L^2)}^2 \notag\\
&&\leq2\|\Lambda_0^b\omega^{(0)}\|_{L^2}^2+8\nu^{-\frac13}\||D_x|^\frac23\Lambda_t^b\theta\|_{L_t^2(L^2)}^2
+ C_1 \|\Lambda_t^b\omega\|_{L_t^\infty(L^2)}\||D_x|^\frac13\Lambda_t^b\omega\|_{L_t^2(L^2)}^2
\notag \\
&&\quad+ C_1\|\Lambda_t^b\omega\|_{L_t^\infty(L^2)}\|\Lambda_t^b\omega_{\neq}\|_{L_t^2(L^2)}\|D_y\Lambda_t^b\omega\|_{L_t^2(L^2)},\label{S2.5eq1}
\een
\ben
&&\|\Lambda_t^b\theta\|_{L_t^\infty(L^2)}^2+\nu\|D_y\Lambda_t^b \theta\|_{L_t^2(L^2)}^2
+\frac14\nu^\frac13\||D_x|^\frac13\Lambda_t^b\theta\|_{L_t^2(L^2)}^2+\|(-\Delta)^{-\frac12}\Lambda_t^b\theta_{\neq}\|_{L_t^2(L^2)}^2 \notag\\
&&\leq 2\|\Lambda_0^b\theta^{(0)}\|_{L^2}^2+C_2\|\Lambda_t^b\omega\|_{L_t^\infty(L^2)}\||D_x|^\frac13\Lambda_t^b\theta\|_{L_t^2(L^2)}^2 \notag \\
&&\quad+C_2\||D_x|^\frac13\Lambda_t^b\omega\|_{L_t^2(L^2)}\||D_x|^\frac13\Lambda_t^b\theta\|_{L_t^2(L^2)}\|\Lambda_t^b\theta_{\neq}\|_{L_t^\infty(L^2)} \notag\\
&&\quad
+C_2\|\Lambda_t^b\omega_{\neq}\|_{L_t^2(L^2)}\|\Lambda_t^b\theta\|_{L_t^\infty(L^2)}\big(\|D_y\Lambda_t^b\theta\|_{L_t^2(L^2)}+\|\Lambda_t^b\theta_{\neq}\|_{L_t^2(L^2)}\big) \label{S2.5eq2}
\een
and
\ben
&&\||D_x|^\frac13\Lambda_t^b\theta\|_{L_t^\infty(L^2)}^2+\nu\|D_y|D_x|^\frac13\Lambda_t^b \theta\|_{L_t^2(L^2)}^2+\frac14\nu^\frac13\||D_x|^\frac23\Lambda_t^b\theta\|_{L_t^2(L^2)}^2
\notag\\
&&\quad+\|(-\Delta)^{-\frac12}|D_x|^\frac13\Lambda_t^b\theta_{\neq}\|_{L_t^2(L^2)}^2\notag\\
&&\leq2\||D_x|^\frac13\Lambda_0^b\theta^{(0)}\|_{L^2}^2+C_3\|\Lambda_t^b\omega\|_{L_t^\infty(L^2)}\||D_x|^\frac23\Lambda_t^b\theta\|_{L_t^2(L^2)}^2 \notag\\
&&\quad+C_3\|\Lambda_t^b\omega_{\neq}\|_{L_t^2(L^2)}\big(\|D_y|D_x|^\frac13\Lambda_t^b\theta\|_{L_t^2(L^2)}+\|D_y\Lambda_t^b\theta\|_{L_t^2(L^2)}\big)\||D_x|^\frac13\Lambda_t^b\theta\|_{L_t^\infty(L^2)} \notag\\
&&\quad+C_3\||D_x|^\frac13\Lambda_t^b\omega\|_{L_t^2(L^2)}\||D_x|^\frac23\Lambda_t^b\theta\|_{L_t^2(L^2)}\||D_x|^\frac13\Lambda_t^b\theta\|_{L_t^\infty(L^2)}.\label{S2.5eq3}
\een

\vskip .1in
With these {\it a priori} bounds at our disposal, our final step is to
prove Theorem \ref{Non2} via the bootstrap argument. We assume  that the initial data $(\omega^{(0)},\theta^{(0)})$ satisfies
\begin{equation*}
\|\omega^{(0)}\|_{H^b}\leq \varepsilon\nu^\beta,\quad \|\theta^{(0)}\|_{H^b}\leq\varepsilon\nu^\alpha,\quad \||D_x|^\frac13\theta^{(0)}\|_{H^b}\leq\varepsilon\nu^\delta,
\end{equation*}
where $\varepsilon>0$ is sufficiently small, and $\beta,\alpha,\delta$ are constants satisfying
\begin{equation}\label{S2.5eqbeta}
\beta\geq\frac23,\qquad \delta\geq\beta+\frac13,\qquad \alpha\geq\delta-\beta+\frac23.
\end{equation}
The bootstrap argument starts with the ansatz that, for $T\le \infty$, the
solution $(\om, \theta)$ of (\ref{vt2}) satisfies
\ben
&&\|\Lambda_t^b\omega\|_{L_T^\infty(L^2)}+\nu^\frac12\|D_y\Lambda_t^b\omega\|_{L_T^2(L^2)}
+\nu^\frac16\| |D_x|^\frac13\Lambda_t^b\omega\|_{L_T^2(L^2)} \notag\\
&& \qquad\qquad \qquad\qquad \qquad\qquad +\|(-\Delta)^{-\frac12}\Lambda_t^b\omega_{\neq}\|_{L_T^2(L^2)}\leq C\varepsilon\nu^\beta, \label{S2.5eq5}\\
&&\|\Lambda_t^b\theta\|_{L_T^\infty(L^2)}+\nu^\frac12\|D_y\Lambda_t^b \theta\|_{L_T^2(L^2)}
+\nu^\frac16\||D_x|^\frac13\Lambda_t^b\theta\|_{L_T^2(L^2)} \notag\\
&& \qquad\qquad\qquad\qquad \qquad\qquad  +\|(-\Delta)^{-\frac12}\Lambda_t^b\theta_{\neq}\|_{L_T^2(L^2)}\leq C\varepsilon\nu^\alpha,\label{S2.5eq6}
\\
&&\||D_x|^\frac13\Lambda_t^b\theta\|_{L_T^\infty(L^2)}+\nu^\frac12\|D_y|D_x|^\frac13\Lambda_t^b \theta\|_{L_T^2(L^2)}+\nu^\frac16\||D_x|^\frac23\Lambda_t^b\theta\|_{L_T^2(L^2)}\notag\\
&&  \qquad\qquad\qquad\qquad \qquad\qquad  +\|(-\Delta)^{-\frac12}|D_x|^\frac13\Lambda_t^b\theta_{\neq}\|_{L_T^2(L^2)}\leq \widetilde C\varepsilon\nu^\delta.\label{S2.5eq7}
\een
The constants $\varepsilon>0, C, \widetilde C>0$ are suitably selected and will be specified later. Making use of the bounds in (\ref{S2.5eq1}), (\ref{S2.5eq2}) and (\ref{S2.5eq3}), we show that (\ref{S2.5eq5}), (\ref{S2.5eq6}) and (\ref{S2.5eq7}) actually holds  with $C$ replaced by $C/2$ and $\widetilde C$ replaced by $\widetilde C/2$. The bootstrap argument then implies that
$T=+\infty$ and (\ref{S2.5eq5}), (\ref{S2.5eq6}) and (\ref{S2.5eq7}) holds for all time.

\vskip .1in
In fact, if we substitute the ansatz given by (\ref{S2.5eq5}), (\ref{S2.5eq6}) and (\ref{S2.5eq7}) in the {\it a priori} estimates in (\ref{S2.5eq1}), (\ref{S2.5eq2}) and (\ref{S2.5eq3}), we find
\beno
&&\|\Lambda_t^b\omega\|_{L_t^\infty(L^2)}^2+\nu\|D_y\Lambda_t^b\omega\|_{L_t^2(L^2)}^2
+\frac18\nu^\frac13\| |D_x|^\frac13\Lambda_t^b\omega\|_{L_t^2(L^2)}^2+\|(-\Delta)^{-\frac12}\Lambda_t^b\omega_{\neq}\|_{L_t^2(L^2)}^2 \notag\\
&&\qquad\qquad \leq 2\varepsilon^2\nu^{2\beta}+8\widetilde C^2\varepsilon^2\nu^{2\delta-\frac23}+C_1C^3\varepsilon^3(\nu^{3\beta-\frac13}+\nu^{3\beta-\frac23}),\\
&&\|\Lambda_t^b\theta\|_{L_t^\infty(L^2)}^2+\nu\|D_y\Lambda_t^b \theta\|_{L_t^2(L^2)}^2
+\frac14\nu^\frac13\||D_x|^\frac13\Lambda_t^b\theta\|_{L_t^2(L^2)}^2+\|(-\Delta)^{-\frac12}\Lambda_t^b\theta_{\neq}\|_{L_t^2(L^2)}^2 \notag\\
&& \qquad\qquad \leq 2\varepsilon^2\nu^{2\alpha}+C_2C^3\varepsilon^3(3\nu^{\beta+2\alpha-\frac13}+\nu^{\beta+2\alpha-\frac23}),\\
&& \||D_x|^\frac13\Lambda_t^b\theta\|_{L_t^\infty(L^2)}^2+\nu\|D_y|D_x|^\frac13\Lambda_t^b \theta\|_{L_t^2(L^2)}^2+\frac14\nu^\frac13\||D_x|^\frac23\Lambda_t^b\theta\|_{L_t^2(L^2)}^2
\notag\\
&&\quad \qquad\qquad \qquad\qquad \qquad \qquad\qquad  +\|(-\Delta)^{-\frac12}|D_x|^\frac13\Lambda_t^b\theta_{\neq}\|_{L_t^2(L^2)}^2\notag\\
&&\qquad\qquad \leq 2\varepsilon^2\nu^{2\delta}+C_3C\widetilde C\varepsilon^3(2\widetilde C\nu^{\beta+2\delta-\frac13}+\widetilde C\nu^{2\delta+\beta-\frac23}+C\nu^{\beta+\alpha+\delta-\frac23}).
\eeno
If we recall \eqref{S2.5eqbeta} and choose
$$\widetilde C\geq8,\quad C\geq32\widetilde C,\quad \varepsilon=\min\big(\frac1{128C_1C},\frac{1}{128C_2C},\frac{\widetilde C}{64C_3C}\big),$$
then (\ref{S2.5eq5}-\ref{S2.5eq6}) hold with $C$ replaced by $C/2$ and \eqref{S2.5eq7} holds with $\widetilde C$ replaced by $\widetilde C/2$.
This completes the proof of Theorem \ref{Non2}.
\end{proof}

\vskip .3in
\section*{\bf Acknowledgments}

All the authors are supported by K. C. Wong Education Foundation.
Wu was partially supported by the National Science Foundation of USA under grant DMS 1624146 and the AT\&T Foundation at Oklahoma State University.
Zhang was  partially supported
by NSF of China under Grants   11371347 and 11688101,  and innovation grant from National Center for
Mathematics and Interdisciplinary Sciences.

\vskip .3in
\bibliographystyle{plain}

\end{document}